\documentclass{article}
\usepackage[english]{babel}
\usepackage[T1]{fontenc}
\usepackage[utf8]{inputenc}
\usepackage{amsthm}
\usepackage{amsmath}
\usepackage{amssymb}
\usepackage{graphicx}
\usepackage[colorlinks=true, citecolor=blue]{hyperref}
\usepackage{dsfont}
\usepackage[all]{xy}
\usepackage[top=4cm,bottom=4cm,left=3cm,right=3cm]{geometry}
\usepackage{enumitem}
\usepackage{mathtools}
\usepackage{subcaption}
\usepackage{authblk}
\usepackage{comment}

\newcommand{\R}{\mathds{R}}
\newcommand{\N}{\mathds{N}}
\newcommand{\Z}{\mathds{Z}}

\theoremstyle{plain}

\newtheorem{thm}{Theorem}
\newtheorem{lem}{Lemma}

\newtheorem{rmk}{Remark}

\numberwithin{equation}{section}

\title{On the local stability of the elapsed-time model in terms of
the transmission delay and interconnection strength}
\date{April 2025}

\author[1]{María J. Cáceres\thanks{\textit{E-mail address}: \texttt{\href{mailto:caceresg@ugr.es}{caceresg@ugr.es}}}}
\author[1]{José A. Cañizo \thanks{\textit{E-mail address}: \texttt{\href{mailto:canizo@ugr.es}{canizo@ugr.es}}}}
\author[2,1]{Nicolas Torres\thanks{\textit{E-mail address}:\texttt{\href{mailto:torres@ugr.es}{torres@ugr.es}}}}
\affil[1]{Universidad de Granada, Departamento de Matemática Aplicada \& IMAG, Andalusia, Spain.}
\affil[2]{Université Côte d'Azur, Laboratoire Jean Alexandre Dieudonné - CNRS UMR 7351, Nice, France.}

\begin{document}

\maketitle

\begin{abstract}
The elapsed-time model describes the behavior of interconnected neurons through the time since their last spike.
It is an age-structured non-linear equation in which age corresponds to the elapsed time since the last discharge, and models many interesting dynamics depending on the type of interactions between neurons. We investigate the linearized stability of this equation by considering a discrete delay, which accounts for the possibility of a synaptic delay due to the time needed to transmit a nerve impulse from one neuron to the rest of the ensemble.
We state a stability criterion that allows to determine if a steady state is linearly stable or unstable depending on the delay and the interaction between neurons. Our approach relies on the study of the asymptotic behavior of  related Volterra-type integral equations in terms of theirs Laplace transforms. The analysis is complemented with numerical simulations illustrating the change of stability of a steady state in terms of the delay and the intensity of interconnections.
\end{abstract}

\noindent{\makebox[1in]\hrulefill}\newline

%\tableofcontents

\section{Introduction}

In the context of modeling brain dynamics, %structured equations have proved to be a useful approach to understand how neurons interact and provided many interesting 
many families of partial differential equations (PDEs)
have proven to be a useful approach for understanding how neurons interact. Among these families, %models, 
we have for example the well-known non-linear Fokker-Planck systems  in the context of the integrate-and-fire systems (referred to as the NNLIF models), where neurons are described through the membrane potential by taking into account the internal and external voltage difference. These systems has been studied by many authors \cite{caceres2011analysis,carrillo2015qualitative,caceres2021understanding,carrillo2023noise,caceres2019global,ikeda2022theoretical,caceres2024asymptotic,caceres2024sequence} through different analysis techniques and numerical simulations. For a comprehensive review of other non-linear partial differential equations in neuroscience we refer to \cite{carrillo2025nonlinear}.

We mention the NNLIF model because in this article we extend the techniques developed for this model \cite{caceres2024asymptotic} to a different and another important approach, namely the Elapsed Time (ET) model. It is
%Another important approach, that we will address in this article, is to consider 
a non-linear age-structured equation for an interconnected ensemble of neurons given by the elapsed time since the last discharge.
%, which is known as the elapsed-time equation (ET).  
In these dynamics, neurons are subjected to random discharges so that when they reach the firing potential, they stimulate other neurons to spike and depending on the type of interaction, different possible behaviors of evolution of the brain activity are possible. This model was initially proposed in \cite{pakdaman2009dynamics} 
and subsequently developed by many authors,
as we will explain later. 

We study the non-linear %renewal equation 
ET model, accounting for the synaptic transmission delay, given by

\begin{equation}
\label{eqdelay}
    \begin{cases}
        \partial_t n +\partial_a n + S(a,r(t-d))n=0 & t,a>0,\\
            n(t,a=0)=r(t)\coloneqq\int_0^\infty S(a,r(t-d))n(t,a)\,da& t>0,\\
            n(t=0,a)=n^0(a)& a>0,\\
            r(t)=r^0(t) & t\in[-d,0],
    \end{cases} 
\end{equation}
where the unknown $n(t,a)$ describes the probability density of finding a neuron at time $t$, whose elapsed time since the last spike is $a\ge 0$, which corresponds to the \emph{age} in the terminology of age-structured equation. 
A neuron spikes when its membrane potential reaches
a threshold value, causing the neuron to discharge and
its age to be reset to $0$.
Neurons spike according to the given firing coefficient $S$, which depends on the age $a$ and the discharging flux of neurons $r(t)$. Therefore, the model is non-linear. For this system $r(t)$ describes the total activity of the network.
%, which for this model we consider it as the total activity. 
%When a neuron discharges at time $t$ its age is reset $0$ and for this model we consider that the number of neurons discharge is given by the flux $r(t)$. 
Moreover, we also assume that when a neuron spikes, its transmission to the rest of the neurons takes a given delay $d>0$, which is called the synaptic delay. This means that the activity $r(t)$ depends on the activity of $r(t-d)$ and the firing coefficient $S$, as we see in the boundary condition of $n$ for $a=0$.

Furthermore, we assume that $S$ is increasing with respect to $a$, which means that when the age $a$ increases neurons are more susceptible to discharge. On the other hand, the dependence of $S$ with respect to $r$ determines %the strength of interconnections and 
the type of regime of interacting neurons. When $S$ is increasing with respect to $r$ we say that the network is \emph{excitatory}, which means that under a high activity neurons are more susceptible to discharge. Similarly, when $S$ is decreasing we say that the network is \emph{inhibitory} and we have the opposite effect on the network.

Usually in the literature, the firing rate is written as $S(a,br)$ where $b\ge0$ is called the network connectivity parameter, so that a small $b$ means a weakly interconnected regime and a large $b$ means strong interconnections. We explore this aspect through numerical simulations in Section \ref{numerical}.

%%%% Esto debemos reformularlo, no significa también fuerte? incluir algo como lo del otro artículo  understand the meaning of weak and strong regimes, the firing coefficient is usually written as $S(a,JX)$, where $J\ge0$ is the network connectivity parameter.
%We have avoided this notation to simplify the presentation of the model, but as mentioned above, our proof covers both regimes since in our case the connectivity parameter is inside $X$.
%We note that this assumption is analogous, in some sense, to the ones given in \cite{mischler2018weak}.

A prototypical example of the coefficient $S$ in the literature is the case with an absolute refractory period, which corresponds to
%\begin{eqnarray}
\begin{align}
   \label{S-refrac}
    \tag{HSr} 
  S(a,r) & = \varphi(r)\mathds{1}_{\{a>\sigma\}},\:\varphi:\R^+\to\R^+\,\mbox{smooth, Lipschitz-bounded,}\\
    &  s_0\le\varphi(r)\,\mbox{for some}\, s_0 \, \mbox{and}\nonumber
 \, \sigma>0.
\end{align}
  %\end{eqnarray}
   where the constant $\sigma$ is called the refractory period. This means that neurons are not susceptible to discharge while $a<\sigma$, but once the refractory period is attained, neurons discharge with firing coefficient $\varphi(r)$.

Throughout this article, the following hypothesis about $S(a,r)$ is assumed:
\begin{equation}
  \label{boundS}
  \tag{HS} 
  S  \, \mbox{is Lipschitz-bounded function, such that } 
  s_0\mathds{1}_{\{a>\sigma\}}\le S(a,r)\, \forall a,r\ge 0, 
  \mbox{for some} \, s_0 \, \mbox{and}
 \, \sigma>0.
\end{equation}
We remark that the regularity of $S$ is not crucial for the results presented in this article and the Lipschitz assumption is just made for simplicity. The results are still valid, for example, in the case of \eqref{S-refrac}. Moreover, the same holds when $S$ satisfies Lipschitz estimates involving the integral of respect to $a$, as it was done for example in \cite{mischler2018,mischler2018weak}.

Finally, the initial data is a non-negative function $n^0\in L^1(\R^+)$ and $r^0\in\mathcal{C}[-d,0]$, the total activity in $[-d,0]$. 
If we consider the case with instantaneous transmission information, i.e. $d=0$, the only initial condition is $n(0,a)$ for $a>0$. In this 
case the equation for $r(t)$, if $n(t,a)$ is known,
$r(t) = \int_0^\infty S(a,r(t)) n(t,a)\ da$ may be an ill-posed
problem; see \cite{pakdaman2009dynamics, canizo2019asymptotic,sepulveda2023well}.

The system formally has the following mass-conservation property
\begin{equation}
\label{massconservation}
    \int_0^\infty n(t,a)\,da = \int_0^\infty n^0(a)\,da=1,\qquad\forall t\ge 0,
\end{equation}
since we assume that $n$ is a probability density.

The elapsed-time equation is a mean-field limit of a microscopic model that establishes a bridge of the dynamics of a single neuron with a population-based approach, whose aspects have been investigated in \cite{pham1998activity,ly2009spike,chevallier2015microscopic,chevallier2015mean,quininao2016microscopic,schwalger2019mind}. Moreover, different extensions of the elapsed time model have been studied by incorporating new elements such as the fragmentation equation \cite{pakdaman2014adaptation}, spatial dependence with connectivity kernel in \cite{torres2020dynamics}, a multiple-renewal equation in \cite{torres2022multiple} and a leaky memory variable in \cite{fonte2022long}.

From an analytical point of view, the elapsed-time model has been studied by many authors. The case without delay, i.e. $d=0$, given by the system
\begin{equation}
\label{eqnodelay}
    \begin{cases}
        \partial_t n +\partial_a n + S(a,r(t))n=0 & t,a>0,\\
            n(t,a=0)=r(t)\coloneqq\int_0^\infty S(a,r(t))n(t,a)\,da& t>0,\\
            n(t=0,a)=n^0(a)& a>0,\\
    \end{cases} 
\end{equation}
has been studied with different techniques. Pakdaman et al. \cite{pakdaman2009dynamics,pakdaman2013relaxation,pakdaman2014adaptation} proved exponential convergence to equilibrium for the case with weak nonlinearities, inspired by the entropy method \cite{perthame2006transport,michel2005general}, while in \cite{canizo2019asymptotic} it was done through the semigroup approach with Doeblin's theory \cite{gabriel2018measure,bansaye2020ergodic} and extended to the case of discrete and distributed delay through comparison techniques in \cite{caceres2025comparison}. In addition, aspects of well-posedness were studied in \cite{sepulveda2023well}. The general dynamics beyond the perturbative case is still an open problem. Some partial results in this direction include \cite{torres2021elapsed}, where the existence of periodic solutions with jump discontinuities was established in the case of strong non-linearities.

On the other hand, the case with distributed delay has been previously studied in \cite{mischler2018,mischler2018weak} through an spectral analysis based on the analysis in \cite{mischler2016spectral} for the growth-fragmentation equation. This corresponds to the modified model given by
\begin{equation}
\label{eqdistri}
    \begin{cases}
        \partial_t n +\partial_a n + S(a,X(t))n=0 & t,a>0,\\
            n(t,a=0)=r(t)\coloneqq\int_0^\infty S(a,X(t))n(t,a)\,da& t>0,\\
            X(t)=\int_{-\infty}^t r(s)\alpha(t-s)\,ds,&t>0\\
            n(t=0,a)=n^0(a)& a>0,\\
    \end{cases} 
\end{equation}
where $X(t)$ takes the role of the total activity, and $\alpha\ge0$ with $\int_0^\infty\alpha(s)\,ds=1$ is called the kernel of distributed delay. In this model $X(t)$ depends on the previous states of the system through the convolution of the discharging flux $r(t)$ and $\alpha(t)$. At a formal level, when $\alpha(t)$ approaches in the sense of distributions to the Dirac's mass $\delta(t-d)$, then we formally get $X(t)=r(t-d)$ and we recover Equation \eqref{eqdelay} with a single discrete delay $d$ and Equation \eqref{eqnodelay} in the case of $d=0$. In this regard, Equation \eqref{eqdelay} may be considered as a singular case of the equation with distributed delay \eqref{eqdistri}.

The steady states $(n^*,r^*)$ of Equation \eqref{eqdelay} (and also \eqref{eqnodelay}) are given by the equation:
\begin{equation}
  \label{stationary-equation}
    \begin{cases}
           \partial_a n^* + S(a,r^*)n^*=0 & a>0,\\
            r^*\coloneqq n^*(a=0)=\int_0^\infty S(a,r^*)n^*(a)\,da,&
    \end{cases}
\end{equation}
whose solutions are $n^*(a):=r^*e^{-\int_0^a S(a',r^*) \ da'}$, with $r^*>0$,  such that 
$n^*$ satisfies
the conservation law $\int_0^\infty n^*(a) da = 1$.
This means that $r^*$ should be
a solution of the equation
\begin{equation}
\label{Int-r}
    rI(r)=1,\ \textrm{with}\:I(r)\coloneqq \int_0^\infty e^{-\int_0^a S(s,r)ds}da.
\end{equation}
Inhibitory systems possess a single steady state. In contrast, excitatory systems with a bounded $S$
always have at least one steady state, and may exhibit multiple steady states \cite{pakdaman2009dynamics}.
%In particular, when the system is inhibitory there exists a unique steady state, while in the excitatory case with $S$ bounded, the system has at least one steady state and they might be multiple \cite{pakdaman2009dynamics}.

One notable case in the study of the elapsed-time model is the linear case. When the firing coefficient $S$ depends on the age $a$ and a fixed
$\bar{r}\ge0$, the system  becomes linear:
\begin{equation}
\label{eqlinear}
    \begin{cases}
        \partial_t n +\partial_a n + S(a,\bar{r})n=0 & t,a>0,\\
            n(t,a=0)=r(t)\coloneqq\int_0^\infty S(a,\bar{r})n(t,a)\,da& t>0.\\
            n(t=0,a)=n^0(a)& a>0,
    \end{cases} 
\end{equation}
Its steady state $(n^*,r^*)$ is given by 
\begin{equation}
  \label{steady-state-linear}
n^*(a)=r^*e^{-\int_0^a S(a',\bar{r}) \ da'}
  \quad \mbox{and}
  \quad
  r^*=\frac{1}{\int_0^\infty e^{-\int_0^a S(a',\bar{r})da'}da},
\end{equation}
which satisfies the previously mentioned conservation law, and $r^*$ is the unique solution to $rI(\bar{r})=1$, with $I$ given in \eqref{Int-r}.

Equation \eqref{eqlinear} can be %cast in the form of 
written as an abstract ODE in the space of borelian finite signed measures $\mathcal{M}(\R^+)$, given by
\begin{equation}
    \label{operator-L}
    \partial_t n =L_{\bar{r}}[n]\coloneqq-\partial_a n - S(a,\bar{r})n+\delta_0 \int_0^\infty S(a,\bar{r})n(a)\,da,
\end{equation}
where the domain of the linear operator $L_{\bar{r}}$,
$D(L_{\bar{r}})$, is a dense subset of $\mathcal{M}(\R^+)$ %$D(L_{\bar{r}})\subset\mathcal{M}(\R^+)$ and it is dense
(for a reference on semigroups see for example \cite{engel2000one}).
For the sake of simplicity of the notation in the computations, we treat the elements in $\mathcal{M}(\R^+)$ as if they were integrable functions with corresponding generalization.

We will denote the solution of this linear problem  as $e^{tL_{\bar{r}}}$, which determines a positive and mass-preserving semigroup in $\mathcal{M}(\R^+)$, i.e $e^{tL_{\bar{r}}}$ is a Markov semigroup. The asymptotic behavior of 
the linear system %$e^{tL_{\bar{r}}}$ 
is well-known, and  we describe it by the following theorem.

\begin{thm}[Exponential convergence to the steady state of the linear system]
\label{doeblin-conv}
Assume that $S$ satisfies \eqref{boundS} and consider the
unique positive steady state $(n^*,r^*)$ such that $n^*\in L^1(0,\infty)$ with $\int_0^\infty n^*\,da=1$ given in \eqref{steady-state-linear}. Then
% .  Then the equation \eqref{eqlinear} has a unique positive steady state $(n^*,r^*)$ such that $n^*\in L^1(0,\infty)$ with $\int_0^\infty n^*\,da=1$ and
there exist  constants $C,\mu>0$ such that for all initial condition $n^0\in \mathcal{M}(\R^+)$ the following estimates hold for all $t\ge 0$
   \begin{equation}
     \label{exp-conv-linear-r}
       |r(t)-\langle n^0\rangle r^*|\le C e^{-\mu t}\|n^0-\langle n^0\rangle n^*\|_{TV},
      % \quad\forall t\ge 0.
   \end{equation}
and
   \begin{equation}
   \label{exp-conv-linear-n}
    \|e^{tL_{\bar{r}}} n^0-\langle n^0\rangle n^*\|_{TV}\le Ce^{-\mu t}\|n^0-\langle n^0\rangle n^*\|_{TV},%\quad\forall t\ge0,   
   \end{equation}
   with $\langle n^0\rangle\coloneqq \int_0^\infty n^0 da$.
\end{thm}

\begin{rmk}
\label{rem:spectral-gap}
It is often said that a linear equation exhibits a \emph{spectral gap} when
its solutions decay exponentially toward the equilibrium.
In this context, we will refer to the constant $\mu>0$
provided in the theorem as the \emph{spectral gap parameter} of the linear
operator $L_{\bar{r}}$.
\end{rmk}

This theorem has been proved using different techniques, such as the entropy method \cite{perthame2006transport}, Doeblin's theory \cite{canizo2019asymptotic} and Kato's inequality \cite{mischler2018weak}. The result %exponential convergence
also holds for the nonlinear equation \eqref{eqdelay} under weak perturbations of the linear case, where a unique equilibrium exists, as shown in \cite{pakdaman2009dynamics,mischler2018,canizo2019asymptotic,caceres2025comparison}.

\subsection{Main results of this article}

The aim of this article is to analyze the linearized stability of Equation \eqref{eqdelay} around a steady state $(n^*,r^*)$ in order to establish a stability criterion in terms of the delay $d$ and a constant $A^*$, which depends on the steady state and the firing coefficient $S$. Specifically, we prove the following two results.

\begin{thm}
\label{stability-poles} 
Assume that $S$ satisfies \eqref{boundS} and let $(n^*,r^*)$ be a given steady state. Consider the constant 
$$
A^*\coloneqq\int_0^\infty\partial_r S(a,r^*)n^*(a)\,da,
$$
the measure $G\coloneqq -\partial_r S(a,r^*)n^*+A^*\delta_0$ and the function $h_0(t)\coloneqq\int_0^\infty S(a,r^*)e^{tL_{r^*}}[G](a)\,da$. Then there exists $\alpha_0>0$ such that the Laplace's transform $\widehat{h_0}(z):=\int_0^\infty e^{-zt} h_0(t) \ dt$ is defined for $\Re(z)\ge-\alpha_0$ and the linear stability  of the equilibrium $(n^*,r^*)$ is determined by the function
\begin{equation}
\label{poles-delay}
   \Phi_d(z)\coloneqq1-e^{-zd}(\widehat{h_0}(z)+A^*),
\end{equation}
in the following sense:
\begin{enumerate}
     \item If the function $\Phi_d(z)$ has a zero with positive real part, then $(n^*,r^*)$ is linearly unstable.
    \item If all the zeros of $\Phi_d(z)$ have negative real part, then $(n^*,r^*)$ is linearly asymptotically stable.
\end{enumerate}
\end{thm}

And this theorem allows us to assert the following stability criterion:

\begin{thm}[Linear stability criterion]
\label{stability-criterion}
   Under the same hypothesis and notation as in Theorem \ref{stability-poles} the following assertions hold.
    \begin{enumerate}
        \item[1.] If $\frac{d}{dr}\left.\left(\frac{1}{I(r)}\right)\right|_{r=r^*}>1$, where $I$ is the function in \eqref{Int-r}, then the equilibrium $(n^*,r^*)$ is linearly unstable for any delay $d>0$. 
        \item[2.] If $|A^*|+\int_0^\infty|h_0(t)|\,dt<1$
        %, where $|A^*|$ and  $h_0$ are %the function defined in Theorem \ref{stability-poles},
        %\eqref{coeff-volterra-h}, 
        then the equilibrium $(n^*,r^*)$ is linearly asymptotically stable for any delay $d\ge0$.
        \item[3.] If $|A^*|>1$, then equilibrium $(n^*,r^*)$ is linearly unstable for any delay $d>0$.
        \item[4.] If $|A^*|<1$ and $(n^*,r^*)$ is linearly asymptotically stable for $d=0$, then it is also stable for $d>0$ sufficiently small.
    \end{enumerate}
\end{thm}

Our main results allow us to make the following observations:
\begin{enumerate}
    \item In both theorems, the key elements in determining the linear stability are the constant $A^*$ and the function $h_0$, which depend only on the equilibrium being considered, and the transmission delay $d$.
    \item  Theorem \ref{stability-poles}  provides a criterion for the linear stability of the equilibrium, based on the zeros of a function $\Phi_d$, defined in the complex plane, that depends solely on $A^*$, $h_0$ and $d$.
    Although finding the zeros of this function may be difficult in general, they can be determined numerically.

    \item  Theorem \ref{stability-criterion} offers a straightforward criterion for determining the linear stability.
   \begin{enumerate}
\item Point \textbf{1.} can be expressed in terms of $A^*$ and $h_0$ as follows
$$
\widehat{h_0}(0)+A^*>1,
$$
because $\frac{d}{dr}\left.\left(\frac{1}{I(r)}\right)\right|_{r=r^*}=\widehat{h_0}(0)+A^*$, as we will prove in Section \ref{sec:linearstab} 
(see Lemma \ref{laplace-h}). We opted to use  $I$ as it makes it easier to examine the slope of 
$1/I$ at the equilibrium. Moreover, this slope provides insight into the pseudo-equilibria sequence, which offers a discrete and simplified representation of the nonlinear model, while maintaining good agreement with it. More details can be found in \cite{caceres2024sequence}, where these sequences were originally introduced for the NNLIF model, and in Section \ref{sec:conclusion}, where we rigorously define the analogous sequence for the age-structured model and outline the initial steps of our ongoing research program aimed at developing these methods further.
   
\item We must point out that the stability criterion of Theorem \ref{stability-criterion} is not exhaustive when $|A^*|<1$. Indeed, point \textbf{2.} asserts stability when $|A^*|<1-\int_0^\infty|h_0(t)|\,dt$ for all $d\ge0$, but we do not have information when $1-\int_0^\infty|h_0(t)|\,dt<|A^*|<1$. Furthermore, the stability result of point \textbf{2.} is consistent with the previous results on the exponential convergence to a unique steady state under weak nonlinearities, as it was previously studied in \cite{pakdaman2009dynamics,canizo2019asymptotic} for the case with $d=0$ and in \cite{mischler2018weak,mischler2018,caceres2025comparison,perthame2025strongly} for the case with delay.

On the other hand, point \textbf{4.} ensures stability for $|A^*|<1$, provided that the equilibrium is stable for the system without delay. Indeed, we conjecture that the equilibrium is asymptotically stable when $A^*<1$ for $d=0$, which includes for example the strongly inhibitory regime where exponential convergence to the unique steady state is expected. In particular, we verify the linear stability for a particular firing coefficient $S$ with an absolute refractory period in Subsection \ref{subsec-refractory}.
\end{enumerate} 
\end{enumerate}

The novelty of this article lies in the fact  that the stability criterion helps to shed light on the behavior of the system \eqref{eqdelay} beyond the case of weak perturbations of the linear case and to understand the transition from weak to high connectivity in the nonlinearity involving the firing coefficient $S$. 
%The main difficulty in linearizing the system \eqref{eqdelay} is the dependence of the activity $r(t)$ in terms of the probability $n$, through the function $S$, which requires a more delicate analysis.

%Furthermore, these 
Our results on the stability state a certain parallelism to the case of voltage-structured models like NNLIF equation studied \cite{caceres2024asymptotic}, on which the techniques throughout this article are based. The main difference is that in the NNLIF the activity depends directly on the density and a connectivity parameter, while in our case the activity $r(t)$ depends on the firing coefficient $S$, which is a function, and must be analyzed separately from $n(t,a)$. Therefore, the main difficulty in linearizing the system \eqref{eqdelay} is the dependence of the activity $r(t)$ in terms of the probability $n$, through the function $S$, which requires a more delicate analysis.

The article is organized as follows. In Section \ref{sec:linear},
we remind the linear equation for \eqref{eqdelay} where in the firing coefficient $S(a,r)$, the value of $r$ is fixed and we give a proof through the approach of Volterra integral equations and their Laplace transform, which motivates the use of this method. This corresponds to a classical technique to study these type of equations and it appears naturally in the context of age-structured models \cite{metz1986dynamics,iannelli2017basic}. 
In Section \ref{sec:linearstab}, we analyze the linearized equation of system \eqref{eqdelay} around a steady state and %we 
prove Theorems \ref{stability-poles} and \ref{stability-criterion} by studying a modified version of the Volterra equation that accounts for the terms involving the delay $d$. %Then, we prove the stability criterion for the system and %we continue with 
We conclude that section with an application in the case of an absolute refractory period. In Section \ref{contour-thm} we give the proof of a theorem that gives information on the asymptotic behavior of the modified Volterra equation, which is needed in the proof of the stability criterion, inspired by \cite{caceres2024asymptotic}.
%Finally  in 
Section \ref{numerical} is devoted to showing some examples of bifurcation diagrams in terms of a connectivity parameter, along with numerical solutions of Equation \eqref{eqdelay} to illustrate the dependence of the delay $d$ on the stability  of the steady state.
Finally, in Section \ref{sec:conclusion}, we present the main conclusions of our work and outline several perspectives for future research.

\section{The linear equation revisited}
\label{sec:linear}

The study of the linear equation \eqref{eqlinear} is fundamental to understand the linear stability around a steady state of Equation
\eqref{eqdelay}. In this section, we present an alternative proof of Theorem \ref{doeblin-conv} using a different approach based on the theory of Volterra integral equations of the form
\begin{equation}
    \label{volterra}
    u(t)= g(t)+\int_0^t h(t-s)u(s)\,ds,
\end{equation}
where $g$ and $h$ are two given locally integrable functions. This approach of integral equations has been widely studied in context of age-structured models in biology such as in the works of Metz et al \cite{metz1986dynamics} and Ianelli et al. \cite{iannelli2017basic}. In our case, we write an equation of this type for the 
activity of the system $r(t)$.
Then, if we formally consider Laplace transform in \eqref{volterra},
we obtain the following equality for $\widehat{u}(z):=\int_0^\infty e^{-zs}u(s)  \ ds$,
the Laplace transform of $u(t)$ solution to \eqref{volterra}

\begin{equation}
    \label{volterra-sol}
    \widehat{u}(z)=\frac{\widehat{g}(z)}{1-\widehat{h}(z)}\eqqcolon F(z).
\end{equation}
And the following result provides the asymptotic behavior of the solution to \eqref{volterra},
under some conditions on functions $g$ and $h$, which guaranty its asymptotic expansion in
terms of the poles of funcion $F$, given in \eqref{volterra-sol}.
%and the connection with the Laplace transform we have 

\begin{thm}
\label{asymptotic-volterra}
    Assume that $g,h\colon [0,\infty)\to\R$ are locally integrable functions and of bounded variation on compact sets. Additionally, suppose there exists $\alpha_0\in\R$ such that $g(t)e^{-\alpha_0 t}\in L^\infty(\R^+)$ and $h(t)e^{-\alpha_0 t}\in L^1(\R^+)\cap L^p(\R^+)$ for some $p>1$. 
    %Let $u$ be the solution of Equation \eqref{volterra} and $F$ the function defined in \eqref{volterra-sol}. Then 
        Then:
    \begin{enumerate}
    \item  The function $F(z)$,  defined in \eqref{volterra-sol}, 
      is a meromorphic function for $\Re(z)>\alpha_0$ with a finite number of poles $\{\lambda_k\}_{k=1}^M$ in this region.
      \item The solution $u(t)$ of Equation \eqref{volterra} has the following asymptotic expansion when $t\to\infty$
    \begin{equation}
    \label{asympt-expansion-linear}
        u(t)=\sum_{k=1}^{M} e^{\lambda_k t}p_k(t)+ O(e^{\alpha t}),
    \end{equation}
    where $\alpha>\alpha_0$ and $p_k(t)$ is a polynomial of degree $m_k-1$, with $m_k$ the order of the pole and determined by the following formula 
    $$p_k(t)e^{\lambda_k t}=\operatorname{Res}\left(e^{zt}F(z),\lambda_k\right).$$
    In addition there exists a constant $C_h>0$, depending on $h$, such that coefficients $\{a_j^k\}_{j=0}^{m_k-1}$ of the polynomials $p_k$ satisfy the bound
    $$|a_j^k|\le C_h \|g(t)e^{-\alpha_0 t}\|_\infty.$$
    \end{enumerate}
\end{thm}
    For a proof of this result see \cite{gripenberg1990volterra,diekmann2012delay,caceres2024asymptotic,iannelli2017basic}.
    
    By applying this theorem to the linear equation \eqref{eqlinear}, we get an alternative proof of Theorem \ref{doeblin-conv}. Despite being a more complex argument, the advantage of this approach is that we obtain a more precise rate of exponential convergence for the activity $r(t)$. In fact, the asymptotic expansion \eqref{asympt-expansion-linear} determined by the poles of the function $F$ in \eqref{volterra-sol}, explicitly shows the optimal rate $\mu$ of exponential convergence in Theorem \ref{doeblin-conv} and the influence of other poles as well. In particular, from the spectral gap property we get $\mu=-\Re(z_0)>0$, where $z_0$ is the eigenvalue with the greatest real part of the operator $L_{\bar{r}}$ in \eqref{operator-L} restricted to the space
    \begin{equation}
    \label{space-X}
    \mathcal{X}\coloneqq\left\{\nu\in\mathcal{M}(\R^+)\colon\int_0^\infty \nu(a)\,da=0\right\},
    \end{equation}
    endowed with the norm of the total variation. 
    
    The proof of Theorem \ref{doeblin-conv} with this approach is as follows.

\begin{proof}[Proof of Theorem \ref{doeblin-conv}]
  Without loss of generality, we can replace the initial data $n^0$ with $\tilde{n}^0\coloneqq n^0-\langle n^0\rangle n^*$ by linearity, in order to restrict ourselves to the space $\mathcal{X}$.
  From the method of characteristics, we get the following expression for the solution $n(t,a)$ of Equation \eqref{eqlinear}
  \begin{equation}
	\label{char}
	n(t,a)=\tilde{n}_0(a-t)e^{-\int_0^t S(a-t+s,\bar{r})\,ds}\mathds{1}_{\{a>t\}}+r(t-a)e^{-\int_0^a S(s,\bar{r})\,ds}\mathds{1}_{\{0<a<t\}}.
      \end{equation}
      Then, by multiplying the expression \eqref{char} by
      $S(a,\bar{r})$ and integrating from 0 to infinity, we obtain
      that $r(t)$ in Equation \eqref{eqlinear} satisfies a Volterra
      integral equation as \eqref{volterra} with coefficients
\begin{align}
 \label{coeff-volterra-linear-g}
    g(t)&=\int_0^\infty S(t+s,\bar{r})e^{-\int_{s}^{t+s}S(a,\bar{r})da}\tilde{n}^0(s)\,ds=-\int_0^{\infty}\frac{d}{dt}\left(e^{-\int_{s}^{t+s}S(a,\bar{r})da}\right)\tilde{n}^0(s)\,ds,\\
  h(t)&=S(t,\bar{r})e^{-\int_{0}^{t}S(a,\bar{r})da}=-\frac{d}{dt}\left(e^{-\int_0^t S(a,\bar{r})da}\right).
        \label{coeff-volterra-linear-h}
\end{align}

    Notice that $g,h$ are of bounded variation in compact sets since $S$ is Lipschitz. 

    Let $s_0,\sigma>0$ be the constants from the condition for $S$
    given in \eqref{boundS}.
    We assert that $\widehat{g}(z)$ and $\widehat{h}(z)$ are defined for $\Re(z)>-s_0$, indeed for $\widehat{g}(z)$ we have

 \begin{equation*}
     \begin{split}
         |\widehat{g}(z)|&\le\int_0^\infty e^{-t\Re(z)}\int_0^\infty S(t+s,\bar{r})e^{-\int_{s}^{t+s}S(a,\bar{r})da}|\tilde{n}^0(s)|\,ds\,dt\\
         &\le \|S\|_\infty \int_0^\infty\int_0^\infty e^{-t\Re(z)}e^{-s_0\int_{s}^{t+s}\mathds{1}_{\{a>\sigma\}}da}|\tilde{n}^0(s)|\,dt\,ds\\
         &\le C\|S\|_\infty\int_0^\infty\int_0^\infty e^{-t(\Re(z)+s_0)}|\tilde{n}^0(s)|\,dt\,ds \\
         &\le C\|S\|_\infty\|\tilde{n}^0\|_{TV}\frac{1}{\Re(z)+s_0}, 
     \end{split}
 \end{equation*}
 where $C>0$ is a constant depending on $s_0$ and $\sigma$. Similarly, we get the same assertion for $\widehat{h}(z)$.

 \
 
 Due the fact that $g$ and $h$ are written in terms of the derivative of exponential functions of $S$ (see
 \eqref{coeff-volterra-linear-g}-\eqref{coeff-volterra-linear-h}), we can
  choose $\alpha_0$ with $-s_0<\alpha_0<0$, so that $g(t)e^{-\alpha _0 t}\in L^1(\R^+),\,h(t)e^{-\alpha_0 t}\in L^1(\R^+)\cap L^\infty(\R^+)$ and the poles of the function $F$ (see \eqref{volterra-sol})  defined in the half-plane $\{z\in\mathds{C}\colon\Re(z)\ge\alpha_0\}$ are given by the equation
 \begin{equation}
     \label{poles-linear}
     1-\widehat{h}(z)=z\int_0^\infty e^{-tz}e^{-\int_0^t S(a,\bar{r})da}dt=0.
 \end{equation}
Observe that $z=0$ is a solution of this equation and it is simple. The corresponding coefficient $p_0$ in the asymptotic expansion is given by

\begin{equation*}
p_0=\operatorname{Res}\left(e^{zt}F(z),0\right)=\lim_{z\to0} \frac{ze^{tz}\widehat{g}(z)}{1-\widehat{h}(z)}=\frac{\widehat{g}(0)}{\int_0^\infty e^{-\int_0^s S(a,\bar{r})\,da}ds}=
 \frac{\int_0^\infty \tilde{n}^0(a)\,da}{\int_0^\infty e^{-\int_0^s S(a,\bar{r})\,da}ds}=\langle \tilde{n}^0\rangle r^*=0.
\end{equation*}

Finally we show that any other possible pole of $F$ has necessarily negative
real part. We prove it by contradiction. We assume that there exists $z=\beta+i\omega$ with $z\neq0$ and $\beta\ge0$ such that $z$ is a solution of Equation \eqref{poles-linear}. This means $\widehat{h}(z)=1$:
%By integrating by parts in this equation, we get that
\begin{equation*}
    \begin{split}
        \int_0^\infty e^{-tz}S(t,\bar{r})e^{-\int_0^t S(a,\bar{r})da}\,dt=1,
    \end{split}
\end{equation*}
and by taking real part we deduce that
$$1=\int_0^\infty \!\!\!\!\!\!
\cos(\omega t)e^{-\beta t}S(t,\bar{r})e^{-\int_0^t S(a,\bar{r})da}dt<
\int_0^\infty \!\!\!\!\!\!
S(t,\bar{r})e^{-\int_0^t S(a,\bar{r})da}dt=
\int_0^\infty \frac{d}{dt}\left(-e^{-\int_0^t S(a,\bar{r})da}\right)dt
=
1,$$
which is a contradiction, and therefore all the poles of $F$, other than $z=0$, have negative real part. 

We complete the proof by applying Theorem \ref{asymptotic-volterra}, which gives the following asymptotic expansion for $r(t)$
$$r(t)=\sum_{k=1}^{M}e^{\lambda_k t}p_k(t)+O(e^{\alpha t}),$$
where $\alpha>\alpha_0$ and $\Re(\lambda_k)<0$ is independent of the initial data, since $\widehat{h}$ in Equation \eqref{poles-linear} does not depend on $\tilde{n}^0$. Furthermore, there exists $C_h$ depending on $h$ (i.e. it depends on the function $S$) such that coefficients $\{a_j^k\}_{j=0}^{k}$ of the polynomials $p_k$ satisfy the bound
$$|a_j^k|\le C_h\|e^{-\alpha t} g(t)\|_\infty\le \Tilde{C}_{S} \|\tilde{n}^0\|_{TV},$$
where $\tilde{C}_{S}$ is a constant depending only on the function $S$. 
Hence we conclude that for $r(t)$ there exists $C,\mu>0$ depending only on $S$, such that 
$$|r(t)|\le C e^{-\mu t}\|n^0-\langle n^0\rangle n^*\|_{TV} \quad\forall t\ge 0,$$
and the estimate for $n$ in \eqref{exp-conv-linear-n} readily follows from the formula of the characteristics \eqref{char}.

%\begin{equation}
%    r(t)=\int_0^t S(t-s,r^*)e^{-\int_{0}^{t-s}S(a,r^*)da}r(s)\,ds+\int_0^\infty S(t+s,r^*)e^{-\int_{s}^{t+s}S(a,r^*)da}n^0(s)\,ds.
%\end{equation}

\end{proof}

We finish this section with the following remark regarding the assumptions on $S$.

\begin{rmk} The result  holds even in case $S$ is not bounded from above.
  We used this restriction  only to show that $\widehat{g}(z)$ and
  $\widehat{h}(z)$ are defined for $\Re(z)>-s_0$.
  We can prove that using $s_0\mathds{1}_{\{a>\sigma\}}\le S(a,r)$ $\forall a,r\ge 0$ and obtain
     \begin{equation*}
     \begin{split}
         |\widehat{g}(z)|&\le\int_0^\infty e^{-t\Re(z)}\int_0^\infty S(t+s,\bar{r})e^{-\int_{s}^{t+s}S(a,\bar{r})da}|\tilde{n}^0(s)|\,ds\,dt\\
         &= -\int_0^\infty |\tilde{n}^0(s)|
             \int_0^\infty e^{-t\Re(z)}\frac{d}{dt}\left(e^{-\int_{s}^{t+s}S(a,\bar{r})da}\right)
             \,dt\,ds\\
         &\le C\|\tilde{n}^0\|_{TV}\left(1+\frac{s_0}{\Re(z)+s_0}\right)
       \end{split}
     \end{equation*}
     and
     $$
    |\widehat{h}(z)|\le C\left(1+\frac{s_0}{\Re(z)+s_0}\right).
    $$
  \end{rmk}

\section{Linear stability of the nonlinear system \eqref{eqdelay}}%the delay equation}
\label{sec:linearstab}
Inspired by the ideas of the previous section, we proceed with the study of the local stability of the equilibria of the nonlinear system \eqref{eqdelay} in order to prove Theorems \ref{stability-poles} and \ref{stability-criterion}. The key idea to understand the asymptotic behavior around a steady state is to adapt the arguments of the Laplace transform to the context of systems with a single discrete delay, by means of the linearized system around the equilibrium.

\

We start by considering $(n^*,r^*)$ a steady state of the nonlinear equation \eqref{eqdelay}, and the
linearized system around it:
%corresponds to the following system for $\phi(t,a)\coloneqq n(t,a)-n^*(a)$
\begin{equation}
\label{linearized}
    \begin{cases}
        \partial_t \phi + \partial_a \phi +S(a,r^*)\phi+\partial_r S(a,r^*)n^*(a) r^\phi(t-d) =0& t,a>0\\
        \phi(t,a=0)=r^\phi(t)\coloneqq\int_0^\infty S(a,r^*)\phi(t,a)\,da + A^* r^\phi(t-d)& t>0\\
        \phi(t=0,a)=\phi^0(a)  %\coloneqq n^0(a)-n^*(a)
        &a>0\\
        r^\phi(t)=r^\phi_0(t)%=r^0(t)-r^*
        & t\in[-d,0].
    \end{cases}
\end{equation}
where
\begin{equation}
  \label{A*}
  A^*\coloneqq\int_0^\infty \partial_r S(a,r^*)n^* (a)\,da,
\end{equation}
$\phi^0\in \mathcal{M}(\R^+)$,  with $\int_0^\infty \phi^0(a)\,da=0$ and $r^\phi_0\in\mathcal{C}[-d,0]$.

We note that, in particular, to study the linear stability of the equilibrium $(n^*,r^*)$
we will consider $\phi^0(a)  \coloneqq n^0(a)-n^*(a)$ and $r^\phi_0(t)=r^0(t)-r^*$.
Therefore, from the mass conservation property, we see that $\int_0^\infty \phi(t,a)da=0$ for all $t\ge 0$.
Hence, it is appropriate to consider the space $\mathcal{X}$ previously defined in \eqref{space-X}.

We want to highlight that the parameter $A^*$, given in \eqref{A*}, is crucial for the long-term behavior of the system, as we will
show later. And it can be interpreted as an indicator of the degree of connectivity in the system.
\

In this section, we will study the long-term behavior of the system \eqref{linearized}. In a first step, we will
analyze the eigenfunctions of the problem in the space $\mathcal{X}$. We will notice that this is not enough to give a linear stability criterion. 
So, in the last part of the section, we will obtain  a modified Volterra integral equation for $r^\phi(t)$,
use the Laplace transform and employ the techniques of Section \ref{sec:linear} to provide our linear stability criterion.

  \subsection{Eigenfunctions of the linearized system}

  A first approach to study stability of steady states is to consider eigenfunctions, i.e., solutions of the form $\phi(t,a)=e^{\lambda\, t }\psi(a)$, with $\lambda\in\mathds{C}$ and $\psi\in\mathcal{X}$ to be determined. By replacing this particular form of $\phi$ in Equation \eqref{linearized}, we obtain the following eigenvalue problem

\begin{equation}
\label{eigenvalue-pb}
   \begin{cases}
        \lambda \psi + \psi'(a) +S(a,r^*)\psi+\partial_r S(a,r^*)n^*(a) e^{-\lambda d}\psi(0)=0& a>0\\
        \psi(a=0)=\int_0^\infty S(a,r^*)\psi(a)\,da + A^*e^{-\lambda d}\psi(0).&\\
    \end{cases}  
\end{equation}
Without loss of generality we can assume the normalization $\psi(0)=1$ (otherwise we obtain $\psi\equiv 0$) so that the eigenfunction is given by the expression

\begin{equation}
\label{eigenfunction}
 \psi(a)=e^{-\int_0^a S(s,r^*)ds}e^{-\lambda a}-e^{-\lambda d}\int_0^a e^{-\int_{a'}^{a}S(s,r^*)ds}e^{-\lambda(a-a')}\partial_r S(a',r^*)n^*(a')\,da',   
\end{equation}
and since $\psi\in\mathcal{X}$ and $n^*(a')=r^*e^{-\int_0^{a'} S(y,r^*) \ dy}$, we obtain by integration the following equation for the eigenvalues $\lambda$
\begin{equation}
\label{eigenvalue-eq}
    e^{-\lambda d} r^*\int_0^\infty e^{-\int_0^a S(s,r^*)ds}\int_0^a e^{-\lambda(a-a')}\partial_r S(a',r^*)\,da'\,da=\int_0^\infty e^{-\int_0^a S(s,r^*)ds}e^{-\lambda a}\,da.
  \end{equation}
  Reciprocally, a complex number $\lambda$ satisfying Equation \eqref{eigenvalue-eq} is an eigenvalue provided that the function defined \eqref{eigenfunction} is in $L^1(\R^+)$.
% Reciprocally, a complex number $\lambda$ satisfying Equation \eqref{eigenvalue-eq} is an eigenvalue provided that the function $e^{-\int_0^a S(s,r^*)ds}e^{-\lambda a}\in L^1(\R^+)$ and the corresponding eigenfunction in $\mathcal{X}$ is given by the formula in \eqref{eigenfunction}.

In order to understand the linear stability of the steady states of Equation \eqref{eqdelay}, we study the roots of Equation \eqref{eigenvalue-eq}. In this regard, we obtain the following simple lemma concerning the instability of equilibria.

\begin{lem}
\label{linear-unstable}  
    If Equation \eqref{eigenvalue-eq} has a root  $\lambda$ with
    $\Re(\lambda)>0$, then the equilibrium $(n^*,r^*)$ of the system \eqref{eqdelay} is linearly unstable.
\end{lem}

\begin{proof}
  If $\lambda$ is a solution of Equation \eqref{eigenvalue-eq} with $\Re(\lambda)>0$, then $\lambda$ is an eigenvalue whose eigenfunction $\psi$ is given by \eqref{eigenfunction} with $\psi\in L^1(\R^+)$ and $\int_0^\infty\psi(a)\,da=0$. Hence, by taking $\phi^0(a)= n^0(a)-n^*(a) =\psi(a)$ we prove the linear instability of $(n^*,r^*)$.
    % If $\lambda$ is a solution of Equation \eqref{eigenvalue-eq} with $\Re(\lambda)>0$, then $\lambda$ is an eigenvalue whose eigenfunction in $\mathcal{X}$ is given in \eqref{eigenfunction}. Hence, by taking $\phi^0=\psi$ we prove the linear instability of $(n^*,r^*)$.
\end{proof}

\subsection{Reformulation as a modified Volterra integral equation}

The case of asymptotic stability of the linear problem \eqref{linearized}, i.e. when all the eigenvalues of Equation \eqref{eigenvalue-pb} have negative real part, requires a more delicate analysis. In order to address it, %this case,
we write Equation \eqref{linearized} as a perturbation of the linear problem introduced in the previous section, obtaining  %and we obtain the following abstract differential equation

\begin{equation}
\label{equation-phi-linearized}
\partial_t \phi= L[\phi]+r^\phi(t-d)G,
\end{equation}
where $L\coloneqq L_{r^*}$ is the linear operator defined in \eqref{operator-L} with $\bar{r}=r^*$ and the distribution $G$ is defined as:
\begin{equation}
  \label{G}
    G\coloneqq -\partial_r S(a,r^*)n^* + A^*\delta_0.
\end{equation}

Next, by applying  Duhamel's formula, we get the following expressions for $\phi(t,a)$ and  $r^\phi(t)$
\begin{equation}
\label{duhamel-phi-linearized}
\begin{cases}
    \phi(t,a)=e^{tL}\phi^0(a)+\int_0^t e^{(t-s)L}[G](a) r^\phi(s-d)\,ds,\\
    r^\phi(t)=\int_0^\infty S(a,r^*)e^{tL}[\phi^0]da+\int_0^t\left(\int_0^\infty S(a,r^*)e^{(t-s)L}[G](a)\,da\right)r^\phi(s-d)ds+A^*r^\phi(t-d).
\end{cases}    
\end{equation}

And we observe that the function $u(t)=r^\phi(t)$ satisfies a modified convolution equation \eqref{volterra} of the form
\begin{equation}
    \label{volterra-delay}
    u(t)= g(t)+\int_0^t h(t-s)u(s)\,ds+A^*u(t-d)\,\mathds{1}_{\{t>d\}},
\end{equation}
whose coefficients $g$ and $h$ are given by
\begin{align}
\label{coeff-volterra-g}
  \begin{split}
   g(t)&\coloneqq \int_0^\infty S(a,r^*)e^{tL}[\phi^0]da+\int_0^d\int_0^\infty S(a,r^*)e^{(t-s)L}[G](a)(r^0(s-d)-r^*)\,da\,ds\\
    &\quad +A^*(r^0(t-d)-r^*)\mathds{1}_{[0,d]}(t),
   \end{split}\\
  h(t)&\coloneqq h_0(t-d)\,\mathds{1}_{\{t>d\}},\:\textrm{with}\: h_0(t)\coloneqq\int_0^\infty S(a,r^*)e^{tL}[G](a)da=r^G_L(t),
  \label{coeff-volterra-h}
\end{align}
where we denote by $r^G_L$ to the neural activity of the linear
system, given by the linear operator $L_{r^*}$.

By formally applying the Laplace transform
in \eqref{volterra-delay}, we obtain that 
$\widehat{u}(z)=F(z)$,
with   
  \begin{equation}
    \label{f}
    F(z)\coloneqq\frac{\widehat{g}(z)}{1-e^{-zd}(\widehat{h_0}(z)+A^*)}.
  \end{equation}

In order to understand the asymptotic behavior of $u=r^\phi$ we need to prove the analogous version of Theorem \ref{asymptotic-volterra} for the modified Volterra equation \eqref{volterra-delay}. We must remark that we cannot apply directly Theorem \ref{asymptotic-volterra} due to the term arising from the discrete delay in Equation \eqref{eqdelay}. The modified equation could be seen as a convolution equation \eqref{volterra} whose kernel has a singular term. 

In this context, the proof of Theorem \ref{stability-poles} relies on the following theorem.

\begin{thm}
\label{asymptotic-volterra 2}
    Assume that $g,h\colon [0,\infty)\to\R$ are locally integrable functions and of bounded variation on compact sets.
    Additionally, suppose there exists  $\alpha_0\in\R$ such that $g(t)e^{-\alpha_0 t}\in L^\infty(\R^+)$ and $h(t)e^{-\alpha_0 t}\in L^1(\R^+)\cap L^p(\R^+)$ for some $p>1$. 
    Then: \begin{enumerate}
        \item The function $F$, defined in \eqref{f}, is a meromorphic function for $\Re(z)>\alpha_0$ with a finite number of poles $\{\lambda_k\}_{k=1}^M$ when $\Re(z)>\frac{\ln |A^*|}{d}$.
        \item The solution $u$ of the modified Equation \eqref{volterra-delay} has the following asymptotic expansion  when $t\to\infty$
    \begin{equation}
    \label{asymptotic-expan}
        u(t)=\sum_{k=1}^{M} e^{\lambda_k t}p_k(t)+ O(e^{\alpha t}),
    \end{equation}
    where $\alpha>\max\{\alpha_0,\frac{\ln|A^*|}{d}\}$ and $p_k(t)$ is a polynomial of degree $m_k-1$, with $m_k$ the order of the pole and determined by the following formula 
    $$p_k(t)e^{\lambda_k t}=\operatorname{Res}\left(e^{zt}F(z),\lambda_k\right).$$
    In addition there exists a constant $C_h>0$, depending on the function $h$, such that coefficients $\{a_j^k\}_{j=0}^{m_k-1}$ of the polynomials $p_k$ satisfy the bound
    $$|a_j^k|\le C_h \|g(t)e^{-\alpha_0 t}\|_\infty.$$
    \end{enumerate}
%    Let $u$ be the solution of the modified Equation \eqref{volterra-delay} and $F$ the function defined in \eqref{f}. Then $F(z)$ is a meromorphic function for $\Re(z)>\alpha_0$ with a finite number of poles $\{\lambda_k\}_{k=1}^M$ when $\Re(z)>\frac{\ln |A^*|}{d}$, and we have the following asymptotic expansion for $u$ when $t\to\infty$
%   \begin{equation}
%    \label{asymptotic-expan}
%        u(t)=\sum_{k=1}^{M} e^{\lambda_k t}p_k(t)+ O(e^{\alpha t}),
%    \end{equation}
%    where $\alpha>\max\{\alpha_0,\frac{\ln|A^*|}{d}\}$ and $p_k(t)$ is a polynomial of degree $m_k-1$, with $m_k$ the order of the pole and determined by the following formula 
%    $$p_k(t)e^{\lambda_k t}=\operatorname{Res}\left(e^{zt}F(z),\lambda_k\right).$$
%    In addition there exists a constant $C_h>0$, depending on the function $h$, such that coefficients $\{a_j^k\}_{j=0}^{m_k-1}$ of the polynomials $p_k$ satisfy the bound
%    $$|a_j^k|\le C_h \|g(t)e^{-\alpha_0 t}\|_\infty.$$
\end{thm}

We leave the proof details for Section \ref{contour-thm}. From this result, the proof of Theorem \ref{stability-poles} is straightforward as we state in the following argument.

\begin{proof}[\textbf{Proof of Theorem \ref{stability-poles}}]
    If $\Phi_d$ defined in \eqref{poles-delay} has a zero with positive real part we observe that from the asymptotic expansion given by Theorem \ref{asymptotic-volterra 2} then $r^\phi(t)$ diverges when $t\to\infty$, which proves that $(n^*,r^*)$ is linearly unstable by choosing $\phi^0=n^0-n^*\not\equiv 0$.

    On the other hand if all the zeros $\Phi_d$ have negative real part, then we observe from Theorem \ref{asymptotic-volterra 2} that $r^\phi(t)\to0$ exponentially when $t\to\infty$. Therefore, we see that $\phi(t,a)$ in Equation \eqref{equation-phi-linearized} verifies that $\|\phi(t,a)\|_{TV}\to 0$ exponentially when $t\to\infty$ thanks to Duhamel's formula in \eqref{duhamel-phi-linearized} and by taking $\phi^0=n^0-n^*$, we get the linear stability of the steady state $(n^*,r^*)$ of the system \eqref{eqdelay}.
\end{proof}

In order to deduce Theorem \ref{stability-criterion} by applying Theorem \ref{stability-poles}, we study in more detail the function $F$ defined in \eqref{f} in the domain of absolute convergence of $\widehat{g}$.

The following technical lemma concerns the exponential decay of $g$ in \eqref{coeff-volterra-g} and the region of absolute convergence $\widehat{g}$.

\begin{lem}
\label{laplace-g}
Let $\mu>0$ %be the constant of the exponential decay in Theorem \ref{doeblin-conv}
be the spectral gap parameter of the linear operator $L_{r^*}$ (see Theorem \ref{doeblin-conv} and
Remark \ref{rem:spectral-gap}),
and  $g$ the function defined in \eqref{coeff-volterra-g}. Thus:
\begin{enumerate}
\item \label{1-g}
  $g(t)e^{-\alpha_0 t}\in L^\infty(\R^+)$ for $\alpha_0\ge-\mu$.
\item \label{2-g}
  $\widehat{g}(z)$ is analytic for $\Re(z)>-\mu$.
\item \label{3-g}
  $g(t)$ is bounded variation on compact sets, if $r^0$ is of bounded variation in $[-d,0]$.
  \end{enumerate}
\end{lem}

\begin{proof}
  The spectral gap given by Theorem \ref{doeblin-conv} for the linear operator  $L=L_{r^*}$ means that
  there exist $C,\mu>0$ such that for  $\phi^0\in\mathcal{X}$ we have
  $$
  \|e^{tL}\phi^0\|_{TV}\le Ce^{-\mu t}\|\phi^0\|_{TV}.
  $$
  We use that decay to prove \ref{1-g} and \ref{2-g}.
    We write $g(t)=g_1(t)+g_2(t)+g_3(t)$ where the functions $g_i(t)$ are given by
    \begin{equation*}
        \begin{split}
            g_1(t)&\coloneqq \int_0^\infty S(a,r^*)e^{tL}[\phi^0](a)\,da\\
            g_2(t)&\coloneqq \int_0^d\int_0^\infty S(a,r^*)e^{(t-s)L}[G](a)(r^0(s-d)-r^*)\,da\,ds\\
            g_3(t)&\coloneqq A^*(r^0(t-d)-r^*)\mathds{1}_{[0,d]}(t),
        \end{split}
    \end{equation*}
    We show \ref{1-g}, proving the statement for each function $g_i$ with $i\in\{1,2,3\}$.
    For $g_1$ we have the following estimate
    $$|g_1(t)|e^{-\alpha_0 t}\le C\|S\|_\infty\|\phi^0\|_{TV}e^{-(\alpha_0+\mu)t},$$
    and similarly for $g_2$ (see \eqref{G} for the definition of $G$), we have the estimate given by
    \begin{equation*}
         \begin{split}
       |g_2(t)|e^{-\alpha_0 t}&\le \|S\|_\infty\|r^0-r^*\| _\infty e^{-\alpha_0 t}\int_0^d \|e^{-(t-s)L}[G]\|_{TV}\,ds\\
       &\le C\|S\|_\infty\|r^0-r^*\| _\infty\|G\|_{TV}e^{-\alpha_0 t}\int_0^d e^{-(t-s)\mu}\,ds\\
       &\le C\|S\|_\infty\|r^0-r^*\| _\infty\|G\|_{TV}\frac{e^{\mu d}-1}{\mu}e^{-(\alpha_0 +\mu)t}.
    \end{split}   
    \end{equation*}
    Observe that $g_3$ is compactly supported and the bound is straightforward. We then conclude that $g(t)e^{-\alpha_0 t}\in L^\infty(\R^+)$ for $\alpha_0\ge -\mu$.
    
    We now prove \ref{2-g}. For  $g_1$ we have the following estimate
    \begin{equation*}
    \begin{split}
       \int_0^\infty\left|g_1(t)e^{-tz}\right|dt&\le \|S\|_\infty\int_0^\infty e^{-t\Re(z)}\|e^{tL}\phi^0\|_{TV}dt\\
       &\le C\|S\|_\infty\|\phi^0\|_{TV}\int_0^\infty e^{-t(\Re(z)+\mu)}dt\\
       &\le C\|S\|_\infty\|\phi^0\|_{TV}\frac{1}{\Re(z)+\mu}<\infty.
    \end{split}
    \end{equation*}
   Similarly for $g_2$, we get the estimate
   \begin{equation*}
   \begin{split}
    \int_0^\infty\left|g_2(t)e^{-tz}\right|dt&\le \|S\|_\infty\|r^0-r^*\|_\infty\int_0^\infty e^{-t\Re(z)}\int_0^d \|e^{(t-s)L}[G]\|_{TV}ds\,dt\\
    &\le  C\|S\|_\infty\|r^0-r^*\|_\infty\|G\|_{TV}\frac{e^{\mu d}-1}{\mu}\int_0^\infty e^{-t(\Re(z)+\mu)}dt\\
    &\le C\|S\|_\infty\|r^0-r^*\|_\infty\|G\|_{TV}\frac{e^{\mu d}-1}{\mu}\frac{1}{\Re(z)+\mu}<\infty.
   \end{split}
    \end{equation*}
    Finally, the Laplace transform of $g_3$ is defined for all $z\in\mathds{C}$ since $g_3$ is compactly supported and therefore $\widehat{g}(z)$ is defined for $\Re(z)>-\mu$.

    To conclude we prove \ref{3-g}. Indeed, we observe that the function $r(t)=\int_0^\infty S(a,r^*)e^{tL}\nu(a)\,da$ is of bounded variation in compacts sets of $t$,  for all $\nu\in\mathcal{M}(\R^+)$, since it is a solution of the convolution equation \eqref{volterra}, whose coefficients given in \eqref{coeff-volterra-linear-g}-\eqref{coeff-volterra-linear-h} are Lipschitz. This implies that $g_1(t),\,g_2(t)$ are of bounded variation in compacts sets and since $r^0$ is of bounded variation in $[-d,0]$, the same conclusion holds for $g_3(t)$ and then for $g(t)$. 
    %$$\int_0^\infty\left|g_3(t)e^{-tz}\right|dt\le |A^*|\,\|r^0-r^*\|_\infty\int_0^d e^{-t\Re(z)}\,dt= |A^*|\,\|r^0-r^*\|_\infty\frac{1-e^{-\Re(z)d}}{\Re(z)}.$$
    
\end{proof}
\begin{rmk}
  We notice that the function
  $$g_1(t)=\int_0^\infty S(a,r^*)e^{tL}[\phi^0](a)\,da=r_L^{\phi_0}(t), $$
corresponds to the neural activity of the linear system, given by the linear operator $L_{r^*}$ with initial condition $\phi^0$. Moreover, $g_2$ can be bounded as follows
$$
g_2(t)\le \Vert r^0-r^*\Vert_\infty
\int_0^d\int_0^\infty S(a,r^*)e^{(t-s)L}[G](a)\,da\,ds=
\Vert r^0-r^*\Vert_\infty \int_0^dr_L^{G}(t-s) \,ds, 
$$
where $r_L^{G}(t)$ is the neural activity of the linear system, given by the linear operator $L_{r^*}$ with initial condition $G$. 

So Lemma \ref{laplace-g} holds under the needed assumptions for $S$ used in Theorem \ref{doeblin-conv} (see remark \ref{rem:spectral-gap}).
\end{rmk}
The following lemma gives us an explicit formula for Laplace transform
of $h_0$ (see \eqref{coeff-volterra-h}),
$\widehat{h_0}$, and its relation with the slope of tangent line of
the function $\frac{1}{I(r)}$ at $r=r^*$, i.e. the slope of tangent line of
the function $\frac{1}{I(r)}$ at the points of intersection with the bisector of the first quadrant (since $r^*I(r^*)=1$).

\begin{lem}
  \label{laplace-h}
  Let $\mu>0$  be the spectral gap parameter of the linear operator $L_{r^*}$
  (see Theorem \ref{doeblin-conv} and Remark \ref{rem:spectral-gap}).
%Let $\mu>0$ the constant from Theorem \ref{doeblin-conv}.
  Then:
  \begin{enumerate}
  \item \label{1-h}  $h_0(t)e^{-\alpha_0 t}\in L^\infty(\R^+)$
    for $\alpha_0\ge-\mu$. %and
  \item \label{2-h}
    $h_0(t)$ is of bounded variation in compact sets.

  \item \label{3-h}    %Furthermore,
    $\widehat{h_0}(z)$ is defined for $\Re(z)>-\mu$,
    in particular $h_0(t)e^{-\alpha_0 t}\in L^1(\R^+)$ for $\alpha_0>-\mu$.
    %and the
    \item \label{4-h} The following formula holds
    \begin{equation}
      \label{h0hat}
        \displaystyle\widehat{h_0}(z)+A^*= \frac{r^*\int_0^\infty e^{-\int_0^a S(s,r^*)ds}\int_0^a e^{-z(a-a')}\partial_r S(a',r^*)\,da'\,da}{\int_0^\infty e^{-\int_0^a S(s,r^*)ds}e^{-za}\,da}.
    \end{equation}
    %    Moreover, for
    \item \label{5-h}
    For the function $I(r)$ defined in \eqref{Int-r} the following equality holds
    \begin{equation}
       \label{h0hat0}
        \frac{d}{dr}\left.\left(\frac{1}{I(r)}\right)\right|_{r=r^*}=\widehat{h_0}(0)+A^*.
    \end{equation}
        \end{enumerate}
\end{lem}

\begin{proof}
  The proof %of the first two assertions on $h_0$
  of \ref{1-h} and \ref{2-h}  are similar to the respective
  one of Lemma \ref{laplace-g}.

  We prove the assertions involving $\widehat{h_0}$. By Theorem \ref{doeblin-conv} applied to the linear
  operator $L=L_{r*}$, we observe that $\widehat{h_0}(z)$ is defined for $\Re(z)>-\mu$ since
 \begin{equation*}
    \begin{split}
      \int_0^\infty\left|h_0(t)e^{-tz}\right|dt = \int_0^\infty\left|r^G_L(t)e^{-tz}\right|dt 
      \le C\|G\|_{TV}\int_0^\infty e^{-t(\Re(z)+\mu)} dt=\frac{C\|G\|_{TV}}{\Re(z)+\mu}. 
      % \|S\|_\infty\int_0^\infty e^{-t\Re(z)}\|e^{tL}G\|_{TV}dt\\
      %  &\le C\|S\|_\infty\|G\|_{TV}\int_0^\infty e^{-t(\Re(z)+\mu)}dt\\
      %  &\le C\|S\|_\infty\|G\|_{TV}\frac{1}{\Re(z)+\mu}.
    \end{split}
    \end{equation*}
    Moreover, from Theorem \ref{doeblin-conv} we get that $\|e^{tL}\|\le e^{-\mu t}$ in the operator norm and hence for $\Re(z)>-\mu$ we can apply the formula for the resolvent to deduce that
    \begin{equation*}
        \begin{split}
            \widehat{h_0}(z)&=\int_0^\infty e^{-tz}\int_0^\infty S(a,r^*)e^{tL}[G]\,da\,dt\\
            &=\int_0^\infty S(a,r^*)\int_0^\infty e^{tL}[G]e^{-tz}\,dt\,da\\
            &=\int_0^\infty S(a,r^*)(zI-L)^{-1}[G](a)\,da,\\
            &=\int_0^\infty S(a,r^*)f(a)\,da,
        \end{split}
    \end{equation*}
    where $f=(zI-L)^{-1}[G]\in D(L)\cap\mathcal{X}$ is the solution
    of the following equation in the sense of distributions
    \begin{equation*}
       f' + (S(a,r^*)+z)f=G+\delta_0\int_0^\infty S(a,r^*)f(a)\,da,
    \end{equation*}
    or equivalently
    \begin{equation*}
     \begin{cases}
         f'(a)+(S(a,r^*)+z)f=-\partial_r S(a,r^*) n^*(a),\\
         f(0)=\int_0^\infty S(a,r^*)f(a)\,da+A^*,\quad\int_0^\infty f(a)\,da=0,\\
     \end{cases}   
    \end{equation*}
    whose solution is given by
    \begin{equation*}
        \begin{dcases}
        f(a)=f(0)e^{-\int_0^a S(s,r^*)ds}e^{-za}-\int_0^a e^{-\int_{a'}^a S(s,r^*)ds}e^{-z(a-a')}\partial_r S(a',r^*)n^*(a')da',\\
       f(0)=\frac{\int_0^\infty\int_0^a e^{-\int_{a'}^a S(s,r^*)ds}e^{-z(a-a')}\partial_r S(a',r^*)n^*(a')da'da}{\int_0^\infty e^{-\int_0^a S(s,r^*)ds}e^{-za}da}.
        \end{dcases}
    \end{equation*}
    Therefore we get the following expression for $\widehat{h_0}(z)$,
    using the initial condition of $f$:
    \begin{equation*}
    \begin{split}
    \displaystyle
        \widehat{h_0}(z)&=\int_0^\infty S(a,r^*)f(a)\,da=f(0)-A^*\\
            %&=\frac{\int_0^\infty\int_0^a e^{-\int_{a'}^a S(s,r^*)ds}e^{-z(a-a')}\partial_rS(a,r^*)n^*(a)da'da-A^*\int_0^\infty e^{-\int_0^a S(s,r^*)ds}e^{-za}da}{\int_0^\infty e^{-\int_0^a S(s,r^*)ds}e^{-za}da},\\
            &=\frac{r^*\int_0^\infty e^{-\int_0^a S(s,r^*)ds}\int_0^a e^{-z(a-a')}\partial_r S(a',r^*)\,da'\,da}{\int_0^\infty e^{-\int_0^a S(s,r^*)ds}e^{-za}\,da}-A^*,
        \end{split}
    \end{equation*}
    and we obtain \eqref{h0hat}. In particular, for $z=0$, since
    $r^*=1/I(r^*)$ with $I$ defined in \eqref{Int-r}, we get
    \begin{equation*}
        \begin{split}
         \widehat{h_0}(0)+A^*&=\frac{\int_0^\infty e^{-\int_0^a S(s,r^*)ds}\int_0^a\partial_r S(a',r^*)\,da'\,da}{I(r^*)\int_0^\infty e^{-\int_0^a S(s,r^*)ds}\,da}.\\
         &= \frac{1}{I(r^*)^2}\int_0^\infty e^{-\int_0^a S(s,r^*)ds}\left(\int_0^a \partial_r S(a',r^*)da'\right)da,
        \end{split}
    \end{equation*}
    and for the derivative of $I$ at $r=r^*$
    \begin{equation*}
    \frac{d}{dr}\left.\left(\frac{1}{I(r)}\right)\right|_{r=r^*}=-\frac{I'(r^*)}{I(r^*)^2}\\
    = \frac{1}{I(r^*)^2}\int_0^\infty e^{-\int_0^a S(s,r^*)ds}\left(\int_0^a \partial_r S(a',r^*)da'\right)da.
    \end{equation*}
    which proves \eqref{h0hat0} and concludes the proof.    %corresponds to the desired result.
\end{proof}

In order to find the zeros of the function $\Phi_d(z)$ defined in
\eqref{poles-delay} for $\Re(z)>-\mu$, we observe that the problem can be equivalently written using the formula \eqref{h0hat} obtained in Lemma \ref{laplace-h} as follows,

\begin{equation}
  \label{poles-delay2}
    e^{zd}=\frac{r^*\int_0^\infty e^{-\int_0^a S(s,r^*)ds}\int_0^a e^{-z(a-a')}\partial_r S(a',r^*)\,da'\,da}{\int_0^\infty e^{-\int_0^a S(s,r^*)ds}e^{-za}\,da}.
  \end{equation}
  Before analyzing the real part of the solutions to Equation
  \eqref{poles-delay2} to study the linear
  stability of the system \eqref{eqdelay}, we provide the following remark.

  \begin{rmk}[Relation between poles and eigenvalues]
    We note that \eqref{poles-delay2}
    is precisely the equation \eqref{eigenvalue-eq} of the eigenvalue problem.
    Therefore, an eigenvalue is a pole of the function $F$, defined in \eqref{f}.
    On the other hand, a pole of $F$, which is zero of the function
    $\Phi_d(z)$ defined in \eqref{poles-delay}, corresponds to an eigenvalue in the domain of convergence of $F$, i.e, in the domain of $\widehat{h_0}(z) $ and $\widehat{g}$.
\end{rmk}

Lemma \ref{laplace-h} gives the key argument to prove the stability criterion \ref{stability-criterion}, involving the constant $A^*$ and the derivative of $\frac{1}{I(r)}$ at $r=r^*$ that determine the asymptotic behavior of the linearized system \eqref{linearized}. With all these necessary lemmas, we now can prove Theorem \ref{stability-criterion}.

\begin{proof}[Proof of Theorem \ref{stability-criterion}]

  \textbf{Proof of point 1.} From Lemma \ref{laplace-h}, the condition of the derivative of
  $\frac{1}{I(r)}$ at $r=r^*$ means that
   $$ \frac{d}{dr}\left.\left(\frac{1}{I(r)}\right)\right|_{r=r^*}=\widehat{h_0}(0)+A^*>1$$
   and thus the function $\Phi_d(z)=1-e^{-zd}\left(\widehat{h_0}(z)+A^*\right)$  (see \eqref{poles-delay})
   satisfies that $\Phi_d(0)<0$. %=1-(\widehat{h_0}(0)+A^*)<0$.
   Since $\Phi_d(x)$ is real and analytic for $x\ge0$ with $\lim_{x\to\infty} \Phi_d(x)=1$, we conclude that there exists $\bar{x}>0$ such that $\Phi_d(\bar{x})=0$ and therefore Equation \eqref{eigenvalue-pb} has a positive eigenvalue. By Theorem \ref{stability-poles}, we get that $(n^*,r^*)$ is linearly unstable.

   \textbf{Proof of point 2.} For the second case observe that when $z$ is purely imaginary, i.e. $z=i\omega$ for $w\in\R$, the following inequality holds for the function $f(z)\coloneqq e^{-zd}(\widehat{h_0}(z)+A^*)$
   \begin{equation*}
     |f(z)|= |e^{-i\omega d}(\widehat{h_0}(i\omega)+A^*)|\\
     \le  \left|\int_0^\infty h_0(t)e^{-it\omega}\,dt\right|+|A^*|\\
     \le \int_0^\infty|h_0(t)|\,dt+|A^*|<1.
   \end{equation*}
   Notice that $f$ is analytic on the half-plane $\{z\colon\Re(z)\ge0\}$ and for $z=\alpha+i\omega$ with $\alpha\ge 0$, we have that
   $$|f(\alpha+i\omega)|\le e^{-\alpha d}(|\widehat{h_0}(\alpha+i\omega)|+|A^*|)$$
   and thus $\displaystyle\limsup_{\|(\alpha,\omega)\|\to\infty}|f(\alpha+i\omega)|\le |A^*|<1$. As $|f(z)|<1$ on the imaginary axis, from the maximum principle we deduce that $|f(z)|<1$ for all $z$ with $\Re(z)\ge0$. This means $|\Phi_d(z)|>0$ in the positive half plane and therefore all its roots have necessarily negative real part. 
   
   Since $|A^*|<1$ we get that $\frac{\ln |A^*|}{d}<0$. Choose $\alpha\in\left(\frac{\ln |A^*|}{d},0\right)$ so that from Theorem \ref{stability-poles} we get that $(n^*,r^*)$ asymptotically stable.

    \textbf{Proof of point 3.} We make the proof for the case $A^*>1$, while the case of $A^*<-1$ has a similar argument. 

    Observe that zeros of the function $f_1(z)\coloneqq 1-A^*e^{-zd}$ are given by
    $$z_k=\frac{\ln A^*}{d}+\frac{2k\pi i}{d},\quad k\in\Z.$$
    For $k\in\Z$ we  consider the disc $D_k$ with center $z_k$ and radius $\frac{\pi}{d}$,
    whose boundary $\partial D_k$ is given by
    $$
    \partial D_k=\left\{z=z_k+\frac{\pi}{d}e^{i\theta}\colon \theta\in [0,2\pi]\right\}
      $$
    and the following inequality holds for $z\in \partial D_k$
    $$
    |f_1(z)|=\left|1-e^{-\pi e^{i\theta}}\right|
    \ge
    \eta\coloneqq \inf_{\theta\in[0,2\pi]}\left|1-e^{-\pi e^{i\theta}}\right|>0.
    $$
    Consider $f_2(z)\coloneqq -\widehat{h_0}(z)e^{-zd}$ and observe that $f_2(z)\to 0$ when $|\Im(z)|\to\infty$ uniformly in the stripe $\{z\in \mathds{C}\colon \Re(z)\in[\alpha,\beta]\}$, since from Riemann-Lebesgue lemma
    we get, denoting by $\mathcal{F}_t[\cdot]$ the Fourier transform, 
    $$\lim_{|\omega|\to\infty}\widehat{h_0}(s+i\omega)=\sqrt{2\pi}\lim_{|\omega|\to\infty}\mathcal{F}_t[e^{-st}h_0(t)](\omega)=0$$
    and the convergence can be made uniform for $s\in[\alpha,\beta]$. We conclude that there exists $k_0\in\N$ such that for all $k\in\Z$ with $|k|>k_0$, the following inequality holds
    $$|f_2(z)|<\eta\quad\forall z\in \partial D_k,$$
    and hence we get
    $$|f_1(z)|>|f_2(z)|\quad\forall z\in \partial D_k.$$
    By Rouché's theorem, we deduce that $f_1(z)+f_2(z)=1-A^*e^{-zd}-h_0(z)e^{-zd}$ has the exactly the same number of zeros that $f_1(z)=1-A^*e^{-zd}$ in $D_k$ and therefore the function \eqref{poles-delay} has infinitely many roots with positive real part. Moreover, when $|\Im(z)|\to\infty$, these roots approach to the zeros of the function $f_1$, which are contained line 
    $$
    \mathcal{L}=\left\{z\in\mathds{C}\colon \Re(z)=\frac{\ln A^*}{d}\right\}
    $$
    We conclude that Equation \eqref{eigenvalue-pb} has an eigenvalue with positive
    real part and then $(n^*,r^*)$ is linearly unstable by Theorem \ref{stability-poles}.

    \textbf{Proof of point 4.} Assume that $|A^*|<1$ and that $(n^*,r^*)$ is linearly stable for $d=0$. This means that the function $\Phi_0$ defined in \eqref{poles-delay} has no zeros in the set $X_\alpha\coloneqq\{z\in\mathds{C}:\Re(z)>-\alpha\}$ for some $\alpha>0$. We prove that this assertion remains true for the function $\Phi_d$ when $d>0$ is small enough.

    Observe that $\widehat{h_0}(z)\to0$ when $\Re(z)\to\infty$, implying that $\Phi_d(z)=1-e^{-zd}(\widehat{h_0}(z)+A^*)\to 1$, so the possible zeros of $\Phi_d$ lie in a vertical stripe. On the other hand, by replicating the argument of point \textbf{3.} of Rouché's theorem, we know that the zeros of $\Phi_d(z)$ approaches the zeros of the function $f_1(z)=1-A^*e^{-zd}$ when $|\Im(z)|\to\infty$, which have negative real part since $|A^*|<1$. Therefore for $0<\beta<\min\{\alpha,-\frac{\ln|A^*|}{d}\}$, the zeros of $\Phi_d$ contained in $X_{\beta}\subset X_{\alpha}$ must be in a compact set. Since $\Phi_d\to\Phi_0$ uniformly in compact sets when $d\to0$ and $\Phi_0$ does not have any zeros in $X_\beta$, we obtain that for $d>0$ small enough $\Phi_d$ does not have any zeros in $X_\beta$ either, in particular with positive real part. Finally, we apply Theorem \ref{stability-poles} as in the proof of point \textbf{2.} to conclude that the equilibrium $(n^*,r^*)$ is linearly asymptotically stable.
\end{proof}

\subsection[Application: System with an absolute refractory period]{Application: System with an absolute refractory period}
\label{subsec-refractory}

In this subsection we apply the stability criterion of Theorem \ref{stability-criterion} when the system has an absolute refractory period, that is, $S$ has the form given in \eqref{S-refrac}:
$$
S(a,r)=\varphi(r)\mathds{1}_{\{a>\sigma\}},
$$    
where the constant $\sigma>0$ is  the refractory period and $\varphi:\R^+\to\R^+$ is a smooth, Lipschitz and uniformly bounded from below. 
This implies that neurons are unable to spike
while $a<\sigma$,  but after the refractory period is completed, the discharge occurs at a
firing coefficient $\varphi(r)$.

In this situation the equation for the steady states \eqref{Int-r} is given by
\begin{equation}
\label{eq-refract}
    r^*\left(\frac{\sigma\varphi(r^*)+1}{\varphi(r^*)}\right)=1,
\end{equation}
and the constant $A^*$ is computed as follows
\begin{equation}
 A^*=\int_0^\infty \partial_rS(a,r^*)n^*(a) \, da=
  r^*\int_0^\infty \partial_rS(a,r^*)e^{-\int_0^a S(a',r^*) \, da'} \, da=
 % r^*\varphi'(r)\int_0^\infty \mathds{1}_{\{a>\sigma\}}e^{-\varphi(r)\int_0^a \mathds{1}_{\{a'>\sigma\}} \, da'} \,  da=
  r^*\frac{\varphi'(r^*)}{\varphi(r^*)}    
\end{equation}
so that the equation $\Phi_d(z)=0$, or equivalently  Equation \eqref{poles-delay2}, reduces to
  \begin{equation}
  \label{poles-refact}
    e^{zd}=\frac{A^*z}{z+\varphi(r^*)-\varphi(r^*)e^{-\sigma z}}.
\end{equation}

We also assert that the characterization of stability through the poles given in Equation \eqref{poles-delay2} and the asymptotic expansion obtained in
Theorem \ref{asymptotic-volterra 2} are also valid in the case of instantaneous transmission of Equation \eqref{eqnodelay}. In the same spirit of Theorem \ref{stability-criterion}, we assert the corresponding stability criterion for the case with absolute refractory period for all $d\ge0$.

\begin{thm}[Stability criterion for the case with absolute refractory period]
\label{stability-refrac} Assume that $S$ satisfies \eqref{S-refrac}. And consider %Given 
a steady state $(n^*,r^*)$ of the non-linear system  \eqref{eqdelay}. Then, the following assertions hold:
    \begin{enumerate}
        \item[1.] If $A^*<1$ (or equivalently $\varphi'(r^*)<\frac{\varphi(r^*)}{r^*}$), then the equilibrium $(n^*,r^*)$ is linearly asymptotically stable for $d=0$. In particular, when $|A^*|<1$ the equilibrium remains linearly asymptotically stable for sufficiently small $d>0$. Moreover, there exists $\ell\in(0,1)$ such that when $|A^*|<\ell$, the equilibrium is linearly asymptotically stable for all $d>0$. On the other hand, when $A^*<-1$ the equilibrium is linearly unstable for all $d>0$.
        %\item[3.] 
        \item[2.] If $1<A^*<1+\sigma\varphi(r^*)$ (or equivalently $\frac{\varphi(r^*)}{r^*}<\varphi'(r^*)<\left(\frac{\varphi(r^*)}{r^*}\right)^2$), then the equilibrium $(n^*,r^*)$ is linearly unstable for all $d\ge0$. Furthermore, when $A^*=1+\sigma\varphi(r^*)$ then the steady state is linearly unstable for $d>0$.
        \item[3.]
        %\item[4.] 
        If $A^*>1+\sigma\varphi(r^*)$ (or equivalently $\varphi'(r^*)>\left(\frac{\varphi(r^*)}{r^*}\right)^2$), then the equilibrium $(n^*,r^*)$ is linearly asymptotically stable for $d=0$ and linearly unstable for $d>0$. 
    \end{enumerate}
\end{thm}

Observe that we stated the hypothesis of the theorem in terms of $A^*,\sigma$ or in terms of $\varphi'(r^*)$, whose equivalence is a consequence of the formula for steady states in \eqref{eq-refract}.
We also highlight that Theorem \ref{stability-refrac} is more comprehensive than Theorem \ref{stability-criterion}, in the sense that it gives information on the stability for any value of $A^*$, except at $A^*=1$.

Theorem \ref{stability-refrac} is consistent with the previous results of \cite{pakdaman2009dynamics,mischler2018weak,mischler2018,torres2021elapsed,perthame2025strongly} on the convergence to steady state when the firing coefficient is written in the form $S(a,br)$ with certain growth conditions when $b\to\infty$, which corresponds to a strongly connected regime. It is important to remark that the criterion does not depend on the monotony of $S$ with respect to $r$, so the results are still valid if the system is inhibitory or excitatory for different values of $r$.

%%%%Habría que explicar lo de strong
We also mention the strongly inhibitory regime, which belongs to the case when $A^*<1$. In the system with absolute refractory period, convergence to the unique equilibrium was proved in \cite{torres2021elapsed} and then improved in \cite{perthame2025strongly}. In this regard, the stability criterion gives us more information on the exponential rate of convergence near the equilibrium through to the study of poles of Equation \eqref{poles-refact}.

\begin{proof}[Proof of Theorem \ref{stability-refrac}]
    \textbf{Proof of point 1.} Assume that $A^*<1$ and $d=0$. 
    In this case, Equation \eqref{poles-refact} is given by
    \begin{equation}
    \label{poles-d0}
    1=f(z)\coloneqq\frac{A^*z}{z+\varphi(r^*)-\varphi(r^*)e^{-\sigma     z}},
    \end{equation}
    and observe that $z=0$ is not a root of \eqref{poles-d0} since
    $$\lim_{z\to 0}f(z)=\frac{A^*}{1+\sigma\varphi(r^*)}<1.$$
    By considering $z=\alpha+i\omega$ with $z\neq 0$, Equation \eqref{poles-d0} is rewritten as the following system
    \begin{equation}
    \label{real-imag}
        \begin{cases}
        (A^*-1)\alpha=\varphi(r^*)(1-e^{-\sigma\alpha}\cos(\sigma\omega)),
        \\
        (A^*-1)\omega=\varphi(r^*)e^{-\sigma\alpha}\sin(\sigma\omega).
        \end{cases}
    \end{equation}
    We assert that is not possible that $\alpha\ge0$. Indeed, if $\alpha>0$ then from the first equation of the system \eqref{real-imag}
    $$0>(A^*-1)\alpha=\varphi(r^*)(1-e^{-\sigma\alpha}\cos(\sigma\omega))\ge 0,$$
    which is a contradiction. If $\alpha=0$ then we obtain that $\cos(\sigma\omega)=1$ and thus $\sin(\sigma\omega)=0$, implying that $w=0$. Therefore any solution $z$ of Equation \eqref{poles-d0} has negative real part and by Theorem \ref{asymptotic-volterra 2} the equilibrium $(n^*,r^*)$ is asymptotically stable.
    For $d>0$ and $|A^*|<1$, the results readily follow from the  points \textbf{2.} and \textbf{4.} of Theorem \ref{stability-criterion}, in particular we can take $\ell=1-\int_0^\infty|h_0(t)|\,dt$. Similarly when $A^*<-1$, the linear instability is obtained from the point \textbf{3.} of Theorem \ref{stability-criterion}.

    \textbf{Proof of points 2.} Consider the case $A^*\in (1,1+\sigma\varphi(r^*))$. If $d>0$ the linear instability readily follows from the point \textbf{3.} of Theorem \ref{stability-criterion}, since $A^*>1$. Assume now that $d=0$ and observe that at the origin we have $\displaystyle\lim_{x\to 0} f(x)<1$, while at infinity we get 
     $$\lim_{x\to+\infty}f(x)=A^*>1.$$
    From continuity of $f$ we obtain that there exists $\bar{x}>0$ such that $f(\bar{x})=1$ and hence it is a solution of Equation \eqref{poles-d0}. By Theorem \ref{stability-poles} we deduce that $(n^*,r^*)$ is linearly unstable for $d=0$.

    The case $A^*=1+\sigma\varphi(r^*)$ readily follows from point \textbf{3.} of Theorem \ref{stability-criterion}.

    \textbf{Proof of points 3.} We now prove the stability result for $A^*>1+\sigma\varphi(r^*)$. Like the previous case, if $d>0$ then $(n^*,r^*)$ is linearly unstable since $A^*>1$.
    
    Assume now $d=0$. As in the proof of the point \textbf{1.} we proceed by contradiction. Consider again $z=\alpha+i\omega$ with $z\neq 0$ and suppose that $\alpha\ge0$. Since the complex solutions of Equation \eqref{poles-d0} are pairs of conjugates, we can assume without loss of generality that $\omega\ge0$. From the system \eqref{real-imag} we deduce the following inequalities
    \begin{equation*}
        \begin{split}
            \sigma\alpha&\le 1-e^{-\sigma\alpha}\cos(\sigma\omega),\\
            \sigma\omega&\le e^{-\sigma\alpha}\sin(\sigma\omega).
        \end{split}
    \end{equation*}
    As $\alpha\ge0$, we deduce that $\sigma\omega\le\sin(\sigma\omega)$ and we obtain that $\omega=0$ since it was assumed to be non negative. From this assertion we conclude that $\sigma\alpha\le 1-e^{-\sigma\alpha}$ and thus $\alpha=0$, which is a contradiction. Hence, every root of Equation \eqref{poles-d0} has strictly negative real part and by Theorem \ref{asymptotic-volterra 2} the equilibrium $(n^*,r^*)$ is asymptotically stable.
    \end{proof}

\section{Asymptotic behavior for convolution Volterra equation with discrete delay}
\label{contour-thm}
This section is devoted to prove Theorem \ref{asymptotic-volterra 2} by adapting the proof of Theorem \ref{asymptotic-volterra} done in \cite{caceres2024asymptotic} for the Fokker-Planck equation. The main difference between these two results lies in handling the terms involving the discrete delay.

Before giving the proof of theorem, we start by reminding the inversion formula of the Laplace transform.
For a classical reference, see for example \cite{doetsch1977introduction}.
\begin{thm}[Inversion formula]
\label{inversion}
Let $f\colon[0,\infty)\to\R$ and $\alpha$ such that $f(t)e^{-\alpha t}\in L^1(\R^+)$. Let $t>0$ be a point where $f$ has bounded variation in some neighborhood of $t$. Then, the following formula
holds, for every $\beta\ge\alpha$,
    $$\lim_{R\to\infty}\frac{1}{2\pi i}\int_{\beta-iR}^{\beta+iR}e^{st}\widehat{f}(s)\,ds=\frac{1}{2}(f(t^+)+f(t^-)),$$
    where the integral is understood as a complex integral along the straight line with real part $\beta$.
    In particular, if $f$ is continuous at $t$, then the formula reduces to
    $$\lim_{R\to\infty}\frac{1}{2\pi i}\int_{\beta-iR}^{\beta+iR}e^{st}\widehat{f}(s)\,ds=f(t).$$
\end{thm}

We continue with the following two necessary lemmas on the function $F$ previously introduced in \eqref{f}%, given by
$$F(z)\coloneqq\frac{\widehat{g}(z)}{1-e^{-zd}(\widehat{h_0}(z)+A^*)},$$
where $g$ and $h_0$ are given in \eqref{coeff-volterra-g} and
\eqref{coeff-volterra-h}, respectively. The first lemma concerns the residue around a pole of $F$.

\begin{lem}[Residue close to a pole]
\label{residue}
 Consider 
 %$F\coloneqq\frac{\widehat{g}(z)}{1-e^{-zd}(\widehat{h_0}(z)+A^*)}$, with 
 $g$ and $h_0$ under the assumptions of Theorem \ref{asymptotic-volterra 2}, with  $\alpha_0\in\R$, 
 such that $F(z)\coloneqq\frac{\widehat{g}(z)}{1-e^{-zd}(\widehat{h_0}(z)+A^*)}$ is a meromorphic function for $\Re(z)>\alpha_0$. 
 And let $z_0$ be a pole of order $m$ in that region ($\Re(z_0)>\alpha_0$).
 Then, there exists a non-trivial polynomial $p(t)$ of degree $m-1$, such that,
    $$\operatorname{Res}(e^{tz}F(z),z_0)=p(t)e^{z_0t}.$$
    Moreover, there exists a constant $C_h>0$, depending on the function $h$, such that, the coefficients of $p(t)$, denoted by
    $\{a_j\}_{j=0}^{m-1}$, are bounded as follows: 
    $$|a_j|\le C_h \|g(t)e^{-\alpha t}\|_\infty.$$
\end{lem}

The proof is essentially the same as  Lemma 3.3 %Lemma 5.1 
in \cite{diekmann2012delay}, but we include it %the proof 
for the sake of completeness.

\begin{proof}[Proof of Lemma \ref{residue}]
    Let $m\in\N$ be the order of the pole $z_0$. We recall that $\operatorname{Res}(e^{zt}F(z),z_0)$ corresponds to the coefficient of the power $-1$ in the Laurent expansion around $z_0$. Then, we consider the Laurent expansions of $F(z)$ and $e^{zt}$ around the pole $z=z_0$: % are given by
    $$F(z)=\sum_{k=-m}^\infty c_k(z-z_0)^k,\quad e^{zt}=e^{z_0t}e^{(z-z_0)t}=\sum_{k=0}^\infty \frac{t^k}{k!}(z-z_0)^k,$$
    with $c_{-m}\neq0$,
    and by performing the multiplication of these two series, we obtain the desired coefficient: % corresponds to
    $$p(t)\coloneqq \sum_{j=0}^{m-1}c_{-j-1}\frac{t^j}{j!}.$$
    Observe that $p(t)$ is polynomial in $t$ of degree $m-1$ since $c_{-m}\neq 0$. 
    
    In order to estimate the coefficients of $p(t)$, we just need find a bound the coefficients $c_k$ with $k=-m,\dots,-1$. Consider a circle $\gamma$ of radius $r$ around the pole $z_0$ such that it does not enclose any other pole of $F$ and for some $\beta>\alpha$ we verify that $\gamma\subset\{z\in\mathds{C}\colon \Re(z)\ge\beta\}$. From Cauchy's formula we deduce that
    \begin{equation*}
        \begin{split}
            |c_k|=\frac{1}{2\pi}\left|\oint_\gamma \frac{F(w)}{(w-z_0)^{k+1}}dw \right|&\le \frac{1}{2\pi r^{k+1}}\oint_\gamma |F(w)|\,dw\\ 
            &\le \frac{\tilde{C}_h}{2\pi r^{k+1}}\oint_\gamma |\widehat{g}(w)|\,dw\\
            &\le \frac{\tilde{C}_h}{r^k}\sup_{\Re(w)\ge\alpha_0}|\widehat{g}(w)|\,dw\\
            &\le \frac{\tilde{C}_h}{r^k(\beta-\alpha)}\|g(t)e^{-\alpha t}\|_\infty,
        \end{split}
    \end{equation*}
    where $\tilde{C}_h\coloneqq\inf_{z\in\gamma}|1-e^{-zd}(\widehat{h_0}(z)+A^*)|$ depends on $h$ and the choice of $r$ and $\beta$. Therefore, the estimate for the coefficients of $p(t)$ readily follows.
\end{proof}

Finally, we need this technical but important result on the integral of $F$ defined in \eqref{f} over a certain line.

\begin{lem}
\label{contour-left}
    Assume the same hypothesis of Theorem \ref{asymptotic-volterra 2} for the functions $g,h$ and the constant $\alpha_0$. Then for $\alpha>\max\{\alpha_0,\frac{\ln |A|^*}{d}\}$, such that $F$ has no poles on the line $\Re(z)=\alpha$, there exists $C>0$ (depending on norms of $g,h$ and independent of $R>0$) such that the following estimate holds
    $$\left|\int_{-R}^R e^{i\omega t}F(\alpha+i\omega)\,d\omega\right|\le C\quad\forall t>0.$$
\end{lem}
    The main difficulty of this result is that we cannot a proiri assert that $F$ is integrable along the line $\Re(z)=\alpha$ to apply Riemann-Lebesgue lemma. In the same way, we cannot directly assert that $F$ is Laplace transform of a function when $\Re(z)=\alpha$. We can only apply the inversion formula when $\Re(z)$ is large enough since we can verify the hypothesis to apply Theorem \ref{inversion}.

\begin{proof}[Proof of Lemma \ref{contour-left}]
 We rewrite the function $F$ in the following way
 $$F(z)=\frac{\widehat{g}(z)\,\widehat{h_0}^N(z) e^{-Nzd}}{(1-e^{-zd}(\widehat{h_0}(z)+A^*))(1-A^*e^{-zd})^N}+\sum_{k=0}^{N-1}\frac{\widehat{g}(z)\widehat{h_0}^k(z)e^{-kzd}}{(1-A^*e^{-zd})^{k+1}}\eqqcolon F_1(z)+F_2(z),$$
 with $N\in\N$, so that we decompose the desired integral as follows
 \begin{equation}
 \label{F1+F2}
 \int_{-R}^R e^{i\omega t}F(\alpha+i\omega)\,d\omega=\int_{-R}^R e^{i\omega t}F_1(\alpha+i\omega)\,d\omega+\int_{-R}^R e^{i\omega t}F_2(\alpha+i\omega)\,d\omega.    
 \end{equation}
 We start by analyzing the first integral. Since $\alpha>\max\{\alpha_0,\frac{\ln |A|^*}{d}\}$ and $F$ has no poles on the line $\Re(z)=\alpha$, there exists $C_1>0$ such that
 $$\frac{1}{(1-e^{-(\alpha+i\omega)d}(\widehat{h_0}(\alpha+i\omega)+A^*))(1-A^*e^{-(\alpha+i\omega)d})^N}\le C_1\quad\forall\omega\in\R,$$
 so that the integral of $F_1$ is estimated for all $t>0$ as follows
 \begin{equation*}
     \begin{split}
       \left|\int_{-R}^R e^{i\omega t}F_1(\alpha+i\omega)\,d\omega\right|&\le \int_{-\infty}^{+\infty} |F_1(\alpha+i\omega)|\,d\omega\\
       &\le C_1 \int_{-\infty}^{+\infty} |\widehat{g}(\alpha+i\omega)|\,|\widehat{h}(\alpha+i\omega)|^N\,d\omega\\
       &\le 2\pi C_1 \int_{-\infty}^{+\infty} |\mathcal{F}[e^{-\alpha t}g](\omega)|\,|\mathcal{F}[e^{-\alpha t}h](\omega)|^N\,d\omega\\
       &\le 2\pi C_1 \|\mathcal{F}[e^{-\alpha t}g]\|_2 \|\mathcal{F}[e^{-\alpha t}h]\|_{2N}^N\\
       &\le C_1 \|e^{-\alpha t}g\|_2 \|e^{-\alpha t}h\|_{p_0}^N,
     \end{split}
 \end{equation*}
 where we have set $h(t)=h_0(t-d)\mathds{1}_{\{t>d\}}$ and applied Hausdorff-Young inequality with $p_0\in[1,2]$ such that $\frac{1}{p_0}+\frac{1}{2N}=1$. From the hypothesis on $h$ we get that $h(t)e^{-\alpha_0 t}\in L^q(\R^+)$ for all $q\in[1,p]$, so by choosing $N$ large enough we get $p_0\in [1,p]$ and hence $\|e^{-\alpha t}h\|_{p_0}<\infty$, since $\alpha>\alpha_0$. Similarly for $g$ we get that $\|e^{-\alpha t}g\|_2<\infty$. 

On the other hand, observe that $F_2$ is the Laplace transform of $f_2(t)\coloneqq u_g(t)+u_g(t)*\sum_{k=1}^{N-1}u_h(t)^{*k}$ where $u_f$ is the solution of the recursion
$$u_f(t)=f(t)+A^*u_f(t-d)\mathds{1}_{\{t>d\}},\quad t>0.$$
By solving step-by-step this recursion, we observe that $u_g,u_f$ are of bounded variation in compact sets. We assert that $v_g\coloneqq u_g(t)e^{-\alpha t}\in L^1(\R^+)$, indeed $v_g$ is solution of the recursion
$$v_g(t)=g(t)e^{-\alpha t}+A^*e^{-\alpha d}v_g(t-d)\mathds{1}_{\{t>d\}},\quad t>0.$$
From the hypothesis on $g$ we know that $|g(t)e^{-\alpha t}|\le B e^{-(\alpha-\alpha_0)t}$ with $B=\|ge^{-\alpha_0 t}\|_\infty$. Since $|A^*|e^{-\alpha d}<1$ we deduce that $v_g(t)\to 0$ exponentially and thus $v_g\in L^1(\R^+)$ and similarly the same conclusion hols for $v_h$. Hence $f_2$ is of bounded variation in compact sets and $f_2(t)e^{-\alpha t}\in L^1(\R^+)$ and $f_2(t)e^{-\alpha t}\to 0$ exponentially when $t\to\infty$.

By applying Theorem \ref{inversion}, we get the following formula
$$\lim_{R\to\infty}\frac{1}{2\pi i}\int_{\alpha-iR}^{\alpha+iR}e^{st}F_2(s)\,ds=\frac{1}{2}(f_2(t^+)+f_2(t^-)),$$
thus we have
$$\left|\lim_{R\to\infty}\int_{-R}^{R} |F_2(\alpha+i\omega)|\,d\omega\right|\le \pi e^{-\alpha t}(|f_2(t^+)|+|f_2(t^-)|).$$
Observe that right-hand side converges to $0$ when $t\to\infty$ and  therefore there exists $C_2>0$ independent of $R>0$ such that
$$\left|\int_{-R}^R e^{i\omega t}F(\alpha+i\omega)\,d\omega\right|\le C_2,\quad\forall t>0,$$
and $C_2$ depends only on norms of $ge^{-\alpha_0 t}$ and $he^{-\alpha_0 t}$. Finally by adding up the estimates for $F_1$ and $F_2$, the desired result readily follows.
\end{proof}

With these results we proceed with the proof of Theorem \ref{asymptotic-volterra 2}.

\begin{proof}[Proof of Theorem \ref{asymptotic-volterra 2}.]
  From the hypothesis for $\alpha_0\in\R$ and the functions $g,h$, we observe that $F$ is meromorphic for $\Re(z)>\alpha_0$. For $\beta$ large enough we may apply Laplace transform to Equation \eqref{volterra-delay} get
  $$\widehat{u}(z)=\frac{\widehat{g}(z)}{1-e^{-zd}(\widehat{h_0}(z)+A^*)}\quad\textrm{for}\:\Re(z)\ge\beta.$$
  Observe that $\widehat{u}(z)=F(z)$ for $\Re(z)\ge\beta$, but a priori we cannot assert that $\widehat{u}$ is defined for $\Re(z)>\alpha_0$.

  For $R>0$ consider the positively oriented curve
  $$\Gamma=[\beta-IR,\beta+iR]\cup[\beta+iR,\alpha+iR]\cup[\alpha+iR, \alpha-iR]\cup[\alpha-iR,\beta+iR].$$
  We notice that $F$ is analytic in the domain enclosed by $\Gamma$, since there is no singularities, and by applying Cauchy's theorem we get for $t>0$
\begin{equation}
  \label{integrals}
  \begin{split}
   2\pi i\sum_{k=1}\operatorname{Res}(F(z)e^{zt},z_k)=\oint_{\Gamma}e^{zt}F(z)\,dz
   =\int_{\beta-iR}^{\beta+iR}e^{zt}F(z)\,dz-e^{iRt}\int_{\alpha}^{\beta}e^{st}F(s+iR)\,ds\\
-\int_{\alpha-iR}^{\alpha+iR}e^{zt}F(z)\,dz+e^{-iRt}\int_{\alpha}^{\beta}e^{st}F(s-iR)\,ds\eqqcolon I_1+I_2+I_3+I_4, 
  \end{split}   
\end{equation}
where the sum is taken over all the poles of $F$ (which are finite for $\Re(z)\ge\alpha$) inside the curve $\Gamma$.
%$\gamma_N$.
In order to obtain the result, we must analyze each integral $I_i$ for $i\in\{1,2,3,4\}$ in terms of $R$.

For $I_1$, observe from Equation \eqref{volterra-delay} that $u(t)\in L ^\infty$ and it is of bounded variation in compact sets since $g,h$ are. Moreover for $\beta>0$ large enough we get $u(t)e^{-\beta t}\in L^1(\R^+)$ and by Theorem \ref{inversion} we deduce that
$$\lim_{R\to\infty} I_1(R)=u(t).$$

We now continue with $I_2$. From Riemann-Lebesgue we get the following limit
 $$\lim_{R\to\infty}\widehat{g}(s+iR)=\lim_{R\to\infty}\sqrt{2\pi}\mathcal{F}_t[e^{-st}g(t)](R)=0,\:\textrm{uniformly in}\:s\in[\alpha,\beta].$$
 Since the poles of $F$ are finite for $\Re(z)\ge\alpha$, we have that $1-e^{-zd}(\widehat{h_0}(z)+A^*)$ has no zeros in the line $[\beta+iR,\alpha+iR]$ for $R$ large enough and hence 
 $$\lim_{R\to\infty}F(s+iR)= 0,\:\textrm{uniformly in}\:s\in[\alpha,\beta],$$
 which implies that $I_2(R)\to0$ when $R\to\infty$ and the same assertion holds for $I_4(R)$ analogously.

For the integral $I_3$, from Lemma \ref{contour-left} there exists $C>0$ (depending on norms of $g,h$ and independent of $R>0$) such that
 $$|I_3(R)|\le C e^{\alpha t},\quad\forall t>0.$$
Concerning the term involving the residues, from Lemma \ref{residue} we know that for every pole $\lambda_k$ of $F$ with order $m_k$ and
%$\Re(\lambda)>\alpha$,
$\Re(\lambda_k)>\alpha$,
the following formula holds
$$\operatorname{Res}(e^{tz}F(z),\lambda_k)=p_k(t)e^{\lambda_k t},$$
where $p_k(t)$ is a polynomial of degree $m_k-1$ with coefficients $\{a_j^k\}_{j=0}^{m-1}$ and we have the following estimate
$$|a_j^k|\le C_h \|g(t)e^{-\alpha t}\|_\infty,$$
with $C_h>0$ a constant depending on $h$. By taking the limit $R\to\infty$ in Equation \eqref{integrals}, the proof of Theorem \ref{asymptotic-volterra 2} readily follows.
\end{proof}

\begin{rmk}
Observe that in the asymptotic expansion given for $u(t)$ in \eqref{asymptotic-expan} of Theorem \ref{asymptotic-volterra 2} we are ignoring the poles with $\Re(z)<\frac{\ln|A^*|}{d}$. If we were able to replicate the proof including these poles in the contour integral of \eqref{integrals}, then we would get a more precise convergence for $u(t)$. However, if we only want to prove that $u(t)$ converges exponentially to $0$, it is sufficient to check that there are no poles with $\Re(z)>0$ and we do not need to include the poles with $\Re(z)<\frac{\ln|A^*|}{d}$ in the contour integral of \eqref{integrals}.   
\end{rmk}

\section{Numerical simulations}
\label{numerical}
In this section, we simulate some numerical examples to validate the stability criterion stated in Theorems  \ref{stability-criterion} and \ref{stability-refrac},
 when connectivity and delay vary, and explore specific cases not addressed by them, complementing the theoretical analysis.

We simulate systems with an absolute refractory period, which means that $S$ takes the form given in \eqref{S-refrac}:
$$
S_b(a,r)=\varphi_b(r)\mathds{1}_{\{a>\sigma\}}.
$$    
We recall that the constant $\sigma>0$ denotes the refractory period.
We will include in $\varphi_b$ a parameter $b\in\R$ which varies and represents the system's connectivity, so that sign of $\varphi_b'$ will be determined by the sign of $b$ and for a given steady state $(n^*,r^*)$ the value of $A^*$ varies in terms of $b$.

\subsection{Bifurcation diagrams}

We examine two examples with different number of steady states and analyze their stability.

\textbf{Example 1: Several equilibria.} For our first example we choose the firing coefficient given by 
\begin{equation}
    \label{S_b1}
    S_b(a,r)= \frac{1}{1+e^{-9br+3.5}}\mathds{1}_{\{a>0.5\}}.
\end{equation}
We observe that $b\in\R$ can be understood as a parameter that represents connectivity, since the system \eqref{eqdelay} is inhibitory when $b<0$ and excitatory when $b>0$.  
In this example the refractory period is $\sigma=0.5$
and $\varphi_b(r)=\frac{1}{1+e^{-9br+3.5}}$.

In this example, the number of stationary solutions changes as the value of $b$ varies. 
For all $b\in\R$ with  $b<b_1\approx0.9313$ and for $b>b_4\approx1.5314$ 
there exists a unique steady, while for $b\in(b_1,b_4)$ the system has three different steady states, see Figure \ref{Bif1}.
In that figure  we show the bifurcation diagram of the activity of equilibrium $r^*$ in terms of the parameter $b$ in the case without delay, \textit{i.e.} $d=0$.
We observe that the behaviour changes within the interval $(b_1,b_4)$, which is why we use this notation, as we need to
consider both $b_2$ and $b_3$ within this range.
\begin{figure}[ht!]
    \centering
    \begin{subfigure}{0.49\textwidth}
    \centering
    \includegraphics[width=\textwidth]{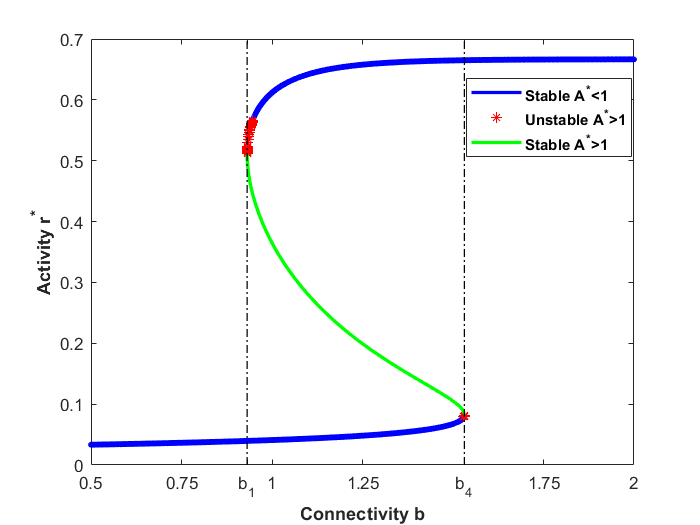}
    \caption{%Main 
    Bifurcation diagram.}
    \label{Bif1}
    \end{subfigure}\\
    \begin{subfigure}{0.49\textwidth}
    \centering
    \includegraphics[width=\textwidth]{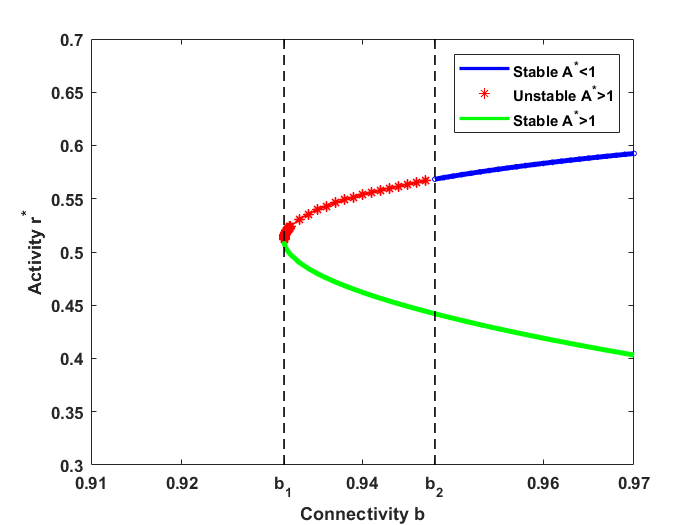}
    \caption{Diagram around the first bifurcation.}
     \label{Bif1b}
    \end{subfigure}
    \begin{subfigure}{0.49\textwidth}
     \includegraphics[width=\textwidth]{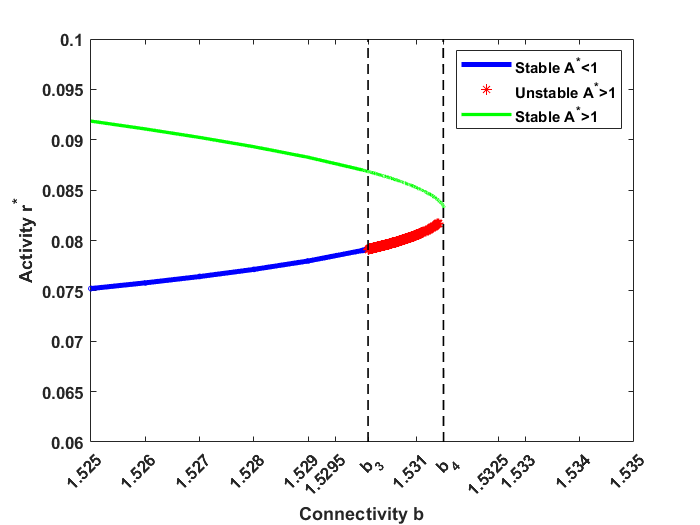}
     \caption{Diagram around the second bifurcation.}
    \label{Bif1c}
    \end{subfigure}
    \caption{Bifurcation diagram for the activity $r^*$  of a stationary solution of \eqref{eqdelay} with $S$ given by \eqref{S_b1}
    in terms of the parameter $b$ with $d=0$.}
\end{figure}
We notice that for all $b\in  (-\infty, b_1) \cup (b_4,\infty)$, $A^*<1$
and then, according to Theorem \ref{stability-refrac}, the unique steady is linearly asymptotically stable.
However, for $b\in(b_1,b_4)$, the system has three different steady states, and their stability vary in terms of $b$. In Figure \ref{Bif1b} we zoom in on the first bifurcation to observe that around the critical value $b=b_1$ two branches of equilibrium arise with $A^*>1$, besides the one with $A^*<1$. For $b\in(b_1,b_2)$ with $b_2\approx0.9480$, one of these steady states with $A^*>1$ is stable, because $A^*>1+\sigma\varphi_b(r^*)$, and the other one is unstable, since $1<A^* < 1+\sigma\varphi_b(r^*)$, according to Theorem \ref{stability-refrac}. In particular at $b=b_2$ we have $A^*=1$ and the unstable branch becomes stable with $A^*<1$ for $b>b_2$. 

Similarly in Figure \ref{Bif1c} we zoom in on the second bifurcation around the critical value $b=b_4$, where we have two branches of equilibrium with $A^*>1$ where one of them is stable, because
 $A^* >1+\sigma\varphi_b(r^*)$, and the other one is unstable for $b\in(b_3,b_4)$ with $b_3\approx1.5301$, since $1<A^* < 1+\sigma\varphi_b(r^*)$. In particular at $b=b_3$ we have $A^*=1$ and the unstable branch becomes stable with $A^*<1$ for $b<b_3$.

This example shows that the stability of a steady state might change without changing the number of equilibriums, which indicates that the dynamics is complex to understand in general. We conjecture that this is due to the existence of periodic solutions induced by the delay $d$ as observed in \cite{sepulveda2023well}. Finally, we remark that when we consider a delay $d>0$, all the branches with $|A^*|>1$ become unstable as a consequence of Theorem \ref{stability-criterion}. This includes the strongly inhibitory case when $A^*<-1$, which is attained when $b$ is negative enough.

\textbf{Example 2: Unique equilibrium.} We now consider a second example where the firing coefficient is given by 
\begin{equation}
    \label{S_b2}
    %S_b(a,r)=\left(\frac{10(br)^2}{(br)^2+1}+0.5\right)\mathds{1}_{\{a>1\}},
    S_{\bar{b}}(a,r)=\left(\frac{10\bar{b}r^2}{\bar{b}r^2+1}+0.5\right)\mathds{1}_{\{a>1\}}, \, \bar{b}=b^2
\end{equation}
In this case the refractory period is $\sigma=1$, $\varphi_{\bar{b}}(r)=\frac{10\bar{b}r^2}{\bar{b}r^2+1}+0.5$, %\frac{10(br)^2}{(br)^2+1}+0.5$,
and we consider $\bar{b}=b^2$  as connectivity parameter,
%$b\ge 0$, 
 which corresponds to an excitatory system. 

\begin{figure}[ht!]
    \centering
    \includegraphics[width=0.5\linewidth]{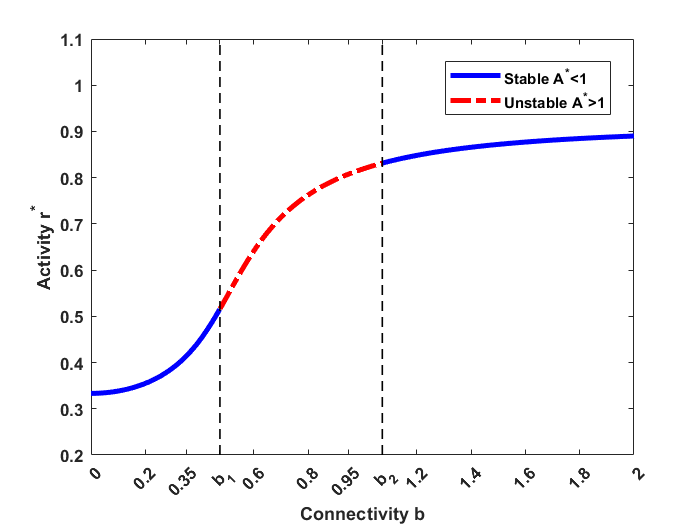}
    \caption{Bifurcation diagram for the activity $r^*$  of a stationary solution of \eqref{eqdelay} with $S$ given by \eqref{S_b2}
    in terms of the parameter $b$ with $d=0$.}
    \label{Bif2}
\end{figure}

In Figure \ref{Bif2} we show the bifurcation diagram of the activity of equilibrium $r^*$ in terms of the parameter $b\ge0$ in the case without delay. In this particular example we have a unique steady state for all $b\in\R$. 

We notice that for $0\le b<b_1\approx0.4750$ and for $b>b_2\approx1.0730$ the steady state is stable with $A^*<1$. For $b\in(b_1,b_2)$ the steady state is unstable with $A^*>1$ and this remains true for all $d>0$, because $1<A^*<1+\sigma\varphi_{\bar{b}}(r^*)$.

As in the previous example, we show another example in which the stability of an equilibrium might change without changing the number of steady states. 

\subsection{Stability for systems with $|A^*|<1$} 

Point 2 of the stability criterion given in Theorem \ref{stability-criterion} states that the equilibrium $(n^*,r^*)$ of the nonlinear system \eqref{eqdelay} is linearly stable whenever $|A^*|<1-\int_0^\infty|h_0(t)|\,dt$ for all $d\ge0$, while point 3 asserts that the equilibrium is linearly unstable for all $d>0$, when $|A^*|>1$. However, the theorem does not give us information when $A^*$ satisfies
\begin{equation}
\label{A-star<1}
 1-\int_0^\infty|h_0(t)|\,dt<|A^*|<1.
\end{equation}
From Theorem \ref{stability-refrac} we know that the steady state is linearly stable for all $d\ge0$ small enough, but, in general, we do not have any information when $d$ is large. 
%We study through two numerical examples that in this case the equilibrium might be stable or unstable depending on the delay 
We investigate this case through two numerical examples, showing that the equilibrium can be either stable or unstable depending on the delay $d$.

If we consider again the previous example with firing coefficient $S_{\bar{b}}$ in \eqref{S_b2} with $b=0.43$  and $\bar{b}=b^2$, we have that
$$r^*\approx 0.4729,\:A^*\approx 0.8500\ \textrm{and}\ \frac{d}{dr}\left.\left(\frac{1}{I(r)}\right)\right|_{r=r^*}\approx0.4481\:,$$
and by applying the equality \eqref{h0hat0} of Lemma \ref{laplace-h}, we get the following inequalities
\begin{equation}
    \begin{split}
     |A^*|+\int_0^\infty|h_0(t)|\,dt\ge&|A^*|+\left|\int_0^\infty h_0(t)\,dt\right|\\
    %\ge   
      =&|A^*|+|\widehat{h_0}(0)|\\
     %\ge
     =&|A^*|+\left|A^*-\tfrac{d}{dr}\left.\left(\tfrac{1}{I(r)}\right)\right|_{r=r^*}\right|\\
     >&1.
    \end{split}
\end{equation}
Hence \eqref{A-star<1} holds and point 2 of the stability criterion given in Theorem \ref{stability-criterion} is not fulfilled. 

For the firing coefficient $S_{\bar{b}}$, we show in Figures \ref{Dominant-1} and \ref{Dominant-2} the real part of the dominant eigenvalue $z_0$, which corresponds to the solution of Equation \eqref{poles-refact} with the greatest real part, in terms of the delay $d$.

\begin{figure}[ht!]
    \centering
    \begin{subfigure}{0.49\textwidth}
    \centering
    \includegraphics[width=\textwidth]{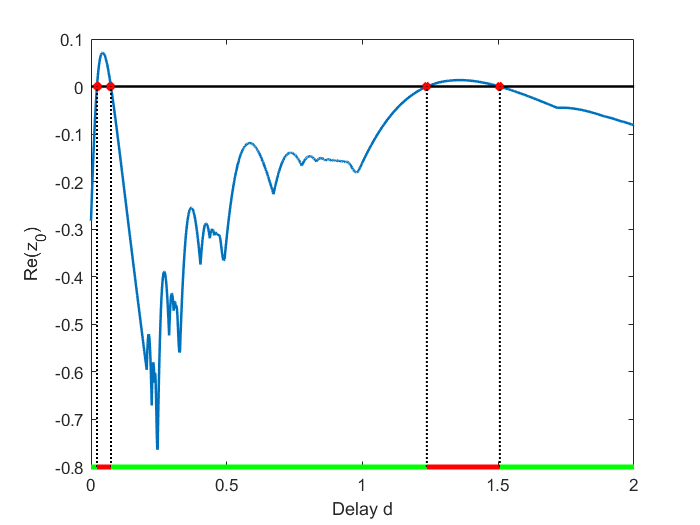}
    \caption{$\Re(z_0)$ for $d\in[0,2]$.}
    \label{Dominant-1}
    \end{subfigure}
    \begin{subfigure}{0.49\textwidth}
    \centering
    \includegraphics[width=\textwidth]{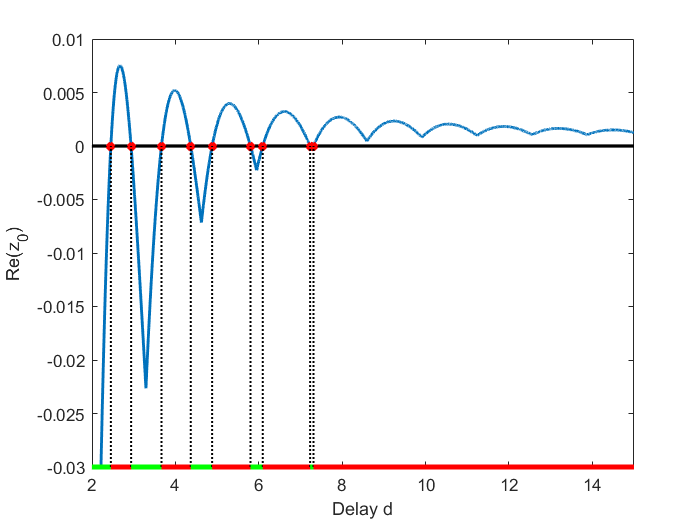}
    \caption{$\Re(z_0)$ for $d\in[2,15]$.}
    \label{Dominant-2}
    \end{subfigure}
    \caption{ 
     Real part of the dominant eigenvalue $z_0$ in terms of the delay $d$ for the firing coefficient $S_{0.43^2}$ given in \eqref{S_b2}. In the axis for $d$, green indicates the values of delay where the equilibrium is linearly stable, while red means instability. The red points indicate the critical values for $d$ where $\Re(z_0)=0$.}
\end{figure}

According to Theorem \ref{stability-poles}, when $\Re(z_0)<0$ the steady state $(n^*,r^*)$ is linearly stable and when $\Re(z_0)>0$ is linearly unstable. We observe in Figures \ref{Dominant-1} and \ref{Dominant-2} that stability changes multiple times when $d$ varies and we observe the existence of some values of $d$ where $\Re(z_0)=0$, where a Hopf bifurcation might arise. The approximate values of these delays and the corresponding eigenvalues are given in Table \ref{tab:hopf-bif}.

\begin{table}[ht!]
    \centering
    \begin{tabular}{|c|c|c|}\hline
        Delay & Eigenvalue & Stability change \\ \hline
        $0.0220$ &  $\pm 5.1718\:i$ & Stable to unstable\\ \hline 
        $0.0735$ &  $\pm 4.3888\:i$ & Unstable to stable\\ \hline
        $1.2369$ &  $\pm 5.1718\:i$ & Stable to unstable \\ \hline
        $1.5051$ &  $\pm 4.3888\:i$ & Unstable to stable \\ \hline
        $2.4518$ &  $\pm 5.1718\:i$ & Stable to unstable\\ \hline
        $2.9368$ &  $\pm 4.3888\: i$ & Unstable to stable\\ \hline
        $3.6667$ &  $\pm 5.1718\: i$ & Stable to unstable\\ \hline
        $4.3684$ & $\pm 4.3888\: i$ & Unstable to stable \\ \hline
        $4.8816$ &  $\pm 5.1718\: i$ & Stable to unstable\\ \hline
        $5.8001$ &  $\pm 4.3888\:i $ & Unstable to stable\\ \hline
        $6.0965$ &  $\pm 5.1718\:i$ & Stable to unstable\\ \hline
        $7.2318$ &  $\pm 4.3888\: i$ & Unstable to stable \\ \hline
        $7.3114$ &  $\pm 5.1718\:i $ & Stable to unstable\\ \hline
    \end{tabular}
    \caption{Approximate critical delays for the firing coefficient $S_{0.43^2}$ given in \eqref{S_b2}, where a Hopf bifurcation arises.}
    \label{tab:hopf-bif}
\end{table}
                                
We observe for $d<d_1\approx0.0220$ that the steady state is linearly stable, which is consistent with point 1. of Theorem \ref{stability-refrac} on the case when $|A^*|<1$ and the delay is small enough. 

On the other hand, we see in Figure \ref{Dominant-2} that for $d>d_{2}\approx 7.3114$ the steady state is unstable. We conjecture that there exists $\ell\in(0,1)$ such that when $\ell<|A^*|<1$ and $d$ is large enough, then the equilibrium is unstable and the real part of the eigenvalues in Equation \eqref{poles-refact} vanishes when $d\to\infty$.

Next, we show numerical simulations to illustrate the asymptotic behavior of Equation \eqref{eqdelay} under the same firing coefficient $S_{0.43^2}$ given by \eqref{S_b2} and different values of $d$.

For $d=0.01$ the dominant eigenvalue is %given by 
$z_0\approx-0.1289+    5.3796\:i$ and hence the steady state $(n^*,r^*)$ is linearly stable. When the system is initialized with the following
slight perturbation of the steady state
\begin{equation}
    \label{init-data1}
   r^0=0.473,\, n^0(a)=0.473e^{-\frac{0.473}{0.527}(a-1)_+},
\end{equation} 
we observe in Figure \ref{Ex1} that $r(t)$ converges to the equilibrium $r^*$ and hence $n(t,a)$ does it to $n^*(a)$ in $L^1(\R^+)$. Notice also that the oscillatory behavior is due to poles with large imaginary part like $z_0$.

\begin{figure}[ht!]
    \centering
    \includegraphics[width=0.49\linewidth]{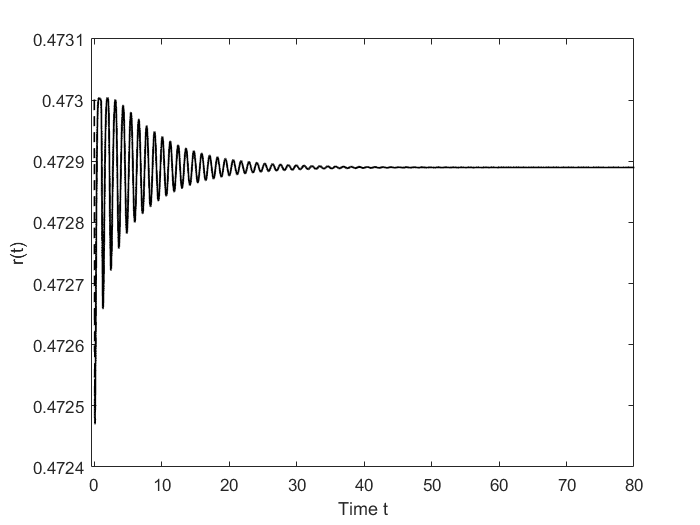}
    \caption{%Numerical solution of $r(t)$ 
    Evolution on time of the activity $r(t)$ of the nonlinear 
    system \eqref{eqdelay} with firing coefficient $S_{0.43^2}$  given in \eqref{S_b2} %with $b=0.43$ 
    and considering
    the initial data \eqref{init-data1} and the transmission delay $d=0.01$.}
    \label{Ex1}
\end{figure}

However, for this firing coefficient the convergence to the steady state is not global, unlike the case of weak nonlinearities \cite{canizo2019asymptotic,mischler2018,mischler2018weak,caceres2025comparison}. Indeed for the initial data given by
\begin{equation}
    \label{init-data2}
    r^0=1,\:n^0(a)=e^{-a},
\end{equation}
we see in Figure \ref{Ex2} that activity $r(t)$ converges towards a periodic solution. We conjecture that when $|A^*|$ is smaller, then the basin of attraction of the steady state is larger.

Moreover, this periodic solution does not have period close to $d$, unlike the periodic solutions observed in \cite{sepulveda2023well} where they appear to be $d$-periodic in the excitatory case for the rate \eqref{S_b2} for $b=1$. A similar phenomenon holds for inhibitory cases, where $2d$-periodic solution were numerically observed, and we expect that the period might change under certain rates $S$.

and $2d$-periodic in the inhibitory case. 

\begin{figure}[ht!]
    \centering
    \includegraphics[width=0.49\linewidth]{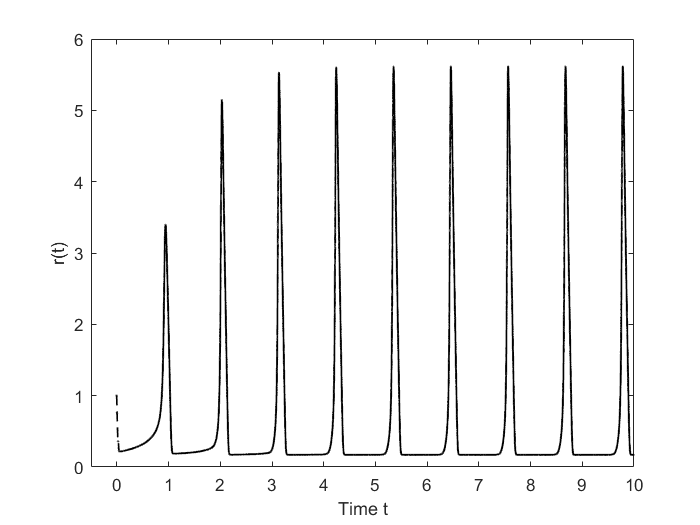}
    \caption{%Numerical solution of $r(t)$ 
    Evolution on time of the activity $r(t)$ of the nonlinear 
    system \eqref{eqdelay} with firing coefficient $S_{0.43²}$  given in \eqref{S_b2} %with $b=0.43$ 
    and considering
    the initial data \eqref{init-data2} and the transmission delay $d=0.01$.}
    \label{Ex2}
\end{figure}

Furthermore, the equilibrium $(n^*,r^*)$ becomes unstable when we increase the delay. Indeed, for $d=0.05$ we notice that Equation \eqref{poles-refact} has $z_0\approx 0.0665+4.6934\:i$ as the dominant eigenvalue and the linear instability follows from Theorem \ref{stability-poles}. Moreover, when the system is initialized with the perturbation of the steady state in \eqref{init-data1}, the activity $r(t)$ does not converge to the equilibrium, as we can observe in Figure \ref{Ex3}. Since $\Re(z_0)$ is small enough, we observe that the solution 
%steady state 
slowly moves away from equilibrium. In fact, the system tends to a periodic profile, as we see in more detail in Figure \ref{Ex3-zoom}, suggesting that it is a consequence of a Hopf bifurcation at $d_1=0.0220$. 

\begin{figure}[ht!]
    \centering
    \begin{subfigure}{0.49\textwidth}
    \centering
    \includegraphics[width=\textwidth]{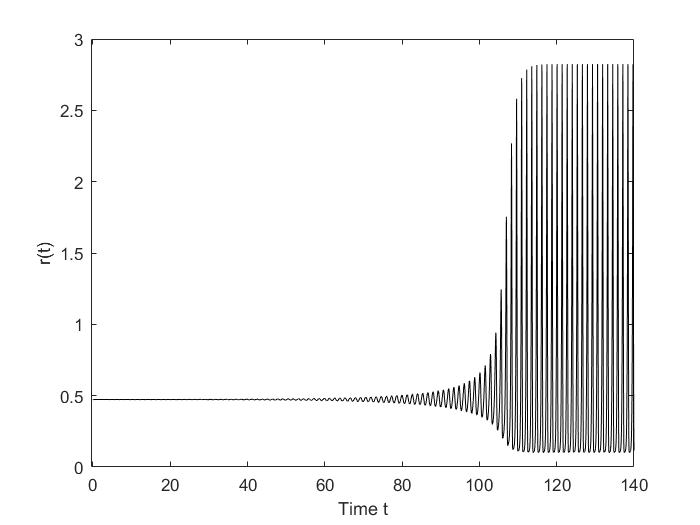}
    \caption{Numerical solution for $t\in[0,140]$.}
    \label{Ex3}
    \end{subfigure}
    \begin{subfigure}{0.49\textwidth}
    \centering
    \includegraphics[width=\textwidth]{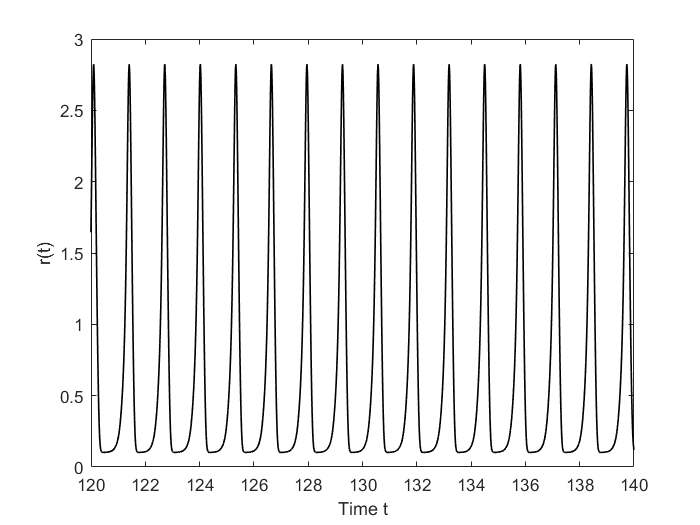}
    \caption{Numerical solution for $t\in[120,140]$.}
    \label{Ex3-zoom}
    \end{subfigure}
    \caption{%Numerical solution of $r(t)$ with delay 
     Evolution on time of the activity $r(t)$ of the nonlinear 
    system \eqref{eqdelay} with firing coefficient $S_{0.43^2}$  given in \eqref{S_b2} %with $b=0.43$ 
    and considering the initial data \eqref{init-data1} and
    the transmission delay $d=0.05$.}
    \label{Ex3-c}
\end{figure}

From these numerical examples, we clearly observe that when the inequality \eqref{A-star<1} holds, the stability strongly depends on the delay $d$ and we conjecture that values found in Table \ref{tab:hopf-bif} generate Hopf bifurcations. We expect a similar behavior in the inhibitory case.

As a final remark, we notice from the numerical simulations of Figures \ref{Ex1}, \ref{Ex2} and \ref{Ex3-c} that jump discontinuities arising at times $t=kd$, with $k\in\N$, tend to vanish when $t\to\infty$ and the asymptotic profile is continuous.

\begin{comment}
\begin{figure}[ht!]
    \centering
    \begin{subfigure}{0.49\textwidth}
    \centering
    \includegraphics[width=\textwidth]{Ex_d2.5_b043.png}
    \caption{Numerical solution for $t\in[0,100]$.}
    \label{Ex3}
    \end{subfigure}
    \begin{subfigure}{0.49\textwidth}
    \centering
    \includegraphics[width=\textwidth]{Ex_d2.5_b043_zoom.png}
    \caption{Numerical solution for $t\in[90,100]$.}
    \label{Ex3-zoom}
    \end{subfigure}
    \caption{%Numerical solution of $r(t)$ with delay 
     Evolution on time of the activity $r(t)$ of the nonlinear 
    system \eqref{eqdelay} with firing coefficient $S_{0.43}$  given in \eqref{S_b2} %with $b=0.43$ 
    and considering the initial data \eqref{init-data1} and
    the transmission delay $d=2.5$.}
\end{figure}
In fact, the system tends to a periodic profile, as we see in more detail in Figure \ref{Ex3-zoom}. Notice also that the highly oscillatory behavior is due to poles with large imaginary part like $z_0$.
\end{comment}

\section{Conclusions and perspectives}
\label{sec:conclusion}
In this work, we have achieved a better understanding of the linear stability of the equilibria of the elapsed-time model by studying the asymptotic behavior of Volterra-type equations. In particular, for a given steady state, we have observed that the constant $A^*$ plays a key role in determining the stability and subsequently the asymptotic behavior of the system. That suggests that $A^*$ may be regarded as an indicator of the strength of interconnections and the type of regime (inhibitory or excitatory,  based on its sign) near the steady state, rather than $\|\partial_r S\|_\infty$ as  considered in previous articles. The stability criterion in Theorems \ref{stability-criterion} and \ref{stability-refrac} complements the previous results in this regard \cite{pakdaman2009dynamics,mischler2018,mischler2018weak,torres2021elapsed}. Theorem \ref{stability-poles} provides a general linear stability criterion based on the poles of the function $\Phi_d$ given
in \eqref{poles-delay}, which depends solely on the firing coefficient
$S$, the transmission delay $d$, and the chosen equilibrium. 
 
%The bifurcation diagrams and numerical simulations show that the dynamics of Equation \eqref{eqdelay} is interesting to analyze. 
The analysis presented in this article might be extended to the case with distributed delay \eqref{eqdistri}.

For a general firing coefficient $S$, we present in Figure \ref{diagram-A*-d} a schematic diagram in terms of $A^*$ and the delay $d$,  summarizing the stability results and conjectures for a given steady state, shedding light on the general dynamics of Equation \eqref{eqdelay}.
\begin{figure}[ht!]
    \centering
    \includegraphics[width=0.65\linewidth]{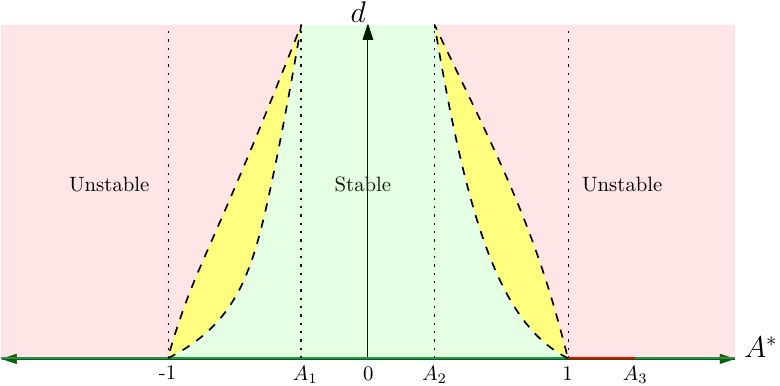}
    \caption{Stability of an equilibrium in terms of $A^*$ and the delay $d$. Green means that the equilibrium is linearly stable, while red means linearly instability and yellow represents a region where multiple changes in the stability arise due to purely imaginary eigenvalues.}
    \label{diagram-A*-d}
\end{figure}
When $|A^*|>1$, we know from the stability criterion of Theorem \ref{stability-criterion} that the steady state is unstable for all $d>0$. Moreover from point 2. of that theorem, we deduce
that there exists two constants $A_1,A_2$ with $-1<A_1<0<A_2<1$ such that the steady is linearly stable when $A^*\in(A_1,A_2)$ and for all delay $d\ge0$. We expect that in this case the equilibrium is the global attractor. 

On the other hand, we conjecture that when $A^*\in(-1,A_1)\cup(A_2,1)$, there exists two critical delays $d_1,\,d_2$ (which depend on $A^*$) such that the equilibrium is asymptotically stable when $d<d_1$, while multiple changes of stability arise due to Hopf bifurcations when $d\in(d_1,d_2)$, and the steady state is unstable when $d>d_2$. When the steady state is asymptotically stable, but the convergence is not global and an attracting periodic orbit exists.

Concerning the existence of periodic solutions, we conjecture that they arise whenever the equilibrium is unstable for $d\ge0$, as was previously observed in \cite{pakdaman2009dynamics,sepulveda2023well,pakdaman2013relaxation,torres2021elapsed} for both strongly excitatory and inhibitory regimes, which motivates the analysis of possible Hopf bifurcations when the value $A^*$ also varies. Moreover, an interesting question is to determine the period and the number of such periodic profiles. We expect that multiple periodic solutions would imply that the elapsed-time system is chaotic under a convenient choice of $S$.

For the case with instantaneous transmission, i.e. $d=0$, we conjecture that the stability criterion in Theorem \ref{stability-refrac} for the case with an absolute refractory period, can be extended to a general firing coefficient $S$. This means that we expect the steady state to be stable when $A^*<1$ and there exists a constant $A_3$ such that the equilibrium is unstable when $A^*\in(1,A_3)$ and stable when $A^*>A_3$.

Furthermore, we need to prove that Equation \eqref{eqdelay} satisfies the linearized stability principle, i.e. the linear dynamics around the steady state determine stability in the nonlinear system, by setting the suitable space and formulation for the delay problem as suggested in \cite{diekmann2012delay,engel2000one}. Numerical simulations in previously cited works suggest that the principle holds.

Finally, another interesting approach to understand the local stability and the asymptotic behavior of \eqref{eqdelay} is by the means of pseudo-equilibrium sequences, whose idea was introduced in  \cite{caceres2024sequence} for the NNLIF model. Notice that in order to solve the delayed equation \eqref{eqdelay},
we can consider time divided into intervals of size $d$ and split the solution $n$ over this partition, such that 
we are left with a sequence of linear problems given by

\begin{equation}
\begin{cases}
    \partial_t n_k + \partial_a n_k + S(a,r_{k-1}(t-d))n_k=0,& t\in I_k, \\
        r_k(t)=n_k(t,a=0)=\int_0^\infty S(a,r_{k-1}(t-d))n_k\,da\\
        n_k((k-1)d^+,a)=n_{k-1}((k-1)d^-,a),\\
        X_0(t)=r^0(t).   
    \end{cases}
\end{equation}
where $k\in\N$ and $I_k\coloneqq ((k-1)d,kd)$ with $n_k\in\mathcal{C}\left(\overline{I_k},L^1(\R^+)\right)$ and $r_k\in\mathcal{C}(\overline{I_k})$.

If the initial data $r^0$ in Equation \eqref{eqdelay} is a constant and the delay $d$ is large enough, the pair $(n_k,r_k)$ should be close to a steady state $(n_k^*,r^*_k)$ in its respective linear problem and from \eqref{steady-state-linear} we obtain the following recurrence

\begin{equation}
r_k^*=\frac{1}{\int_0^\infty e^{-\int_0^a S(a',r^*_{k-1})da'}da}=
\frac{1}{I(r^*_{k-1})}
\quad\mbox{and}\quad n_k^*(a)=r_k^* e^{-\int_0^a S(a',r^*_{k-1})\,da'},
\end{equation}
which resembles the formula for the steady states of the nonlinear problem in \eqref{Int-r}, with the diffe\-ren\-ce
that $I$ is evaluated at $r^*_{k-1}$ instead
of $r^*_{k}$. The sequence $(n_k^*,r_k^*)$ is called the \emph{pseudo-equilibrium sequence} for Equation \eqref{eqdelay}. This
sequence is defined independently of the nonlinear system, by considering the recursive equation $x_k=\frac{1}{I(x_{k-1})}$ to obtain the
\emph{firing rate sequence} $\{r_k^*\}_{k\ge 0}$, and then using it to construct $\{n_k^*\}_{k\ge 0}$ . Observe that the fixed points of that recursive 
equation correspond to the equilibria of the system. In this regard, we expect that asymptotic behavior and local stability is determined by the behavior of the sequence $\{r_k^*\}_{k\ge 0}$ when the delay is large enough under suitable hypothesis, analogously to \cite{caceres2024sequence} for the NNLIF model.
We note that the slope of the function $\frac{1}{I(x)}$ evaluated in its fixed points  determines the long-term behaviour of the recursive model. In this 
sense, point  1. in Theorem \ref{stability-criterion} is agreement with this asymptotic behaviour. 
We aim to address these aspects on a future work.

\subsubsection*{Acknowledgements}

\thanks{\em The authors acknowledge support from projects of the
  Spanish \emph{Ministerio de Ciencia e Innovación} and the European Regional
  Development Fund (ERDF/FEDER) through grants PID2020-117846GB-I00, PID2023-151625NB-100, RED2022-134784-T, and CEX2020-001105-M, funded by\\
  MCIN/AEI/10.13039/501100011033.

NT was supported by the grant Juan de la Cierva FJC2021-046894-I funded by MCIN/AEI and the European Union NextGenerationEU/PRTR.}
\newpage
\bibliography{biblio.bib}
\bibliographystyle{vancouver}

\end{document}